\documentclass{amsart}

\usepackage{amsfonts,amssymb,amsmath,enumerate,verbatim,mathtools,tikz,bm,mathrsfs,tikz-cd,hyperref,comment,stmaryrd,amsthm}
\usepackage{colonequals}
\usepackage{cleveref}
\usepackage{adjustbox}
\usepackage[all,pdf]{xy}
\hypersetup{
    colorlinks=true,
    linkcolor=blue,
    filecolor=magenta,      
    urlcolor=cyan,
}
\usetikzlibrary{arrows}
\usetikzlibrary{decorations.markings}
\tikzset{negated/.style={
        decoration={markings,
            mark= at position 0.5 with {
                \node[transform shape] (tempnode) {$\backslash$};
            }
        },
        postaction={decorate}
    }
}

\makeatletter
\@namedef{subjclassname@2020}{\textup{2020} Mathematics Subject Classification}
\makeatother

\author{Souvik Dey}
\address{Department of Algebra, Faculty of Mathematics and Physics, Charles University in Prague, Sokolovska´ 83, 186 75 Praha, Czech Republic}
\email{souvik.dey@matfyz.cuni.cz}   

\author[Jian Liu]{Jian Liu}
\address{School of Mathematics and Statistics, and Hubei Key Laboratory of Mathematical Sciences,  Central China Normal University,  Wuhan 430079, P.R. China}
\email{jianliu@ccnu.edu.cn}

\author[Liran Shaul]{Liran Shaul}
\address{Department of Algebra, Faculty of Mathematics and Physics, Charles University in Prague, Sokolovsk\'a 83, 186 75 Praha, Czech Republic}
\email{shaul@karlin.mff.cuni.cz}



\newcommand{\sg}{\mathsf{sg}}

\DeclareMathOperator{\amp}{amp}

\DeclareMathOperator{\depth}{depth}

\DeclareMathOperator{\h}{H}

\newcommand{\Z}{\mathbb{Z}}

\newcommand{\T}{\mathcal{T}}

\newcommand{\D}{\mathsf{D}}
\newcommand{\K}{\mathsf{K}}
\newcommand{\C}{\mathcal{C}}

\newcommand{\G}{\mathcal{G}}

\newcommand{\para}{\mathbin{\!/\mkern-5mu/\!}}

\pgfdeclarelayer{bg}    % declare background layer
\pgfsetlayers{bg,main} 
\newcommand{\x}{{\bm{x}}}

\newcommand{\del}{\partial}

\newcommand{\m}{\mathfrak{m}}

\newcommand{\p}{\mathfrak{p}}
\newcommand{\q}{\mathfrak{q}}

\DeclareMathOperator{\pd}{pd}

\DeclareMathOperator{\level}{\mathsf{level}}
\DeclareMathOperator{\Gor}{\mathsf{Gor}}
\DeclareMathOperator{\FID}{\mathsf{FID}}
\DeclareMathOperator{\thick}{\mathsf{thick}}
\DeclareMathOperator{\gin}{\mathsf{gin}}
\DeclareMathOperator{\maxSpec}{maxSpec}
\DeclareMathOperator{\Inj}{Inj}
\DeclareMathOperator{\Rfd}{Rfd}

\DeclareMathOperator{\id}{id}
\DeclareMathOperator{\Spec}{Spec}
\DeclareMathOperator{\Supp}{\mathsf{Supp}}
\DeclareMathOperator{\ac}{ac}

\DeclareMathOperator{\Cone}{Cone}
\DeclareMathOperator{\RHom}{\mathsf{RHom}}
\DeclareMathOperator{\Add}{Add}
\DeclareMathOperator{\seqdepth}{seq.depth}
\DeclareMathOperator{\CM}{CM}
\renewcommand{\k}{\Bbbk}

\def\ann{\operatorname{ann}}

\def\depth{\operatorname{depth}}

\def\Ext{\operatorname{Ext}}

\def\h{\operatorname{H}}
\def\Hom{\operatorname{Hom}}

\def\id{\mathrm{id}}

\def\m{\mathfrak{m}}

\def\p{\mathfrak{p}}

\newtheorem{theorem}{Theorem}[section]
\newtheorem{proposition}[theorem]{Proposition}

\newtheorem{lemma}[theorem]{Lemma}
\newtheorem{corollary}[theorem]{Corollary}

\theoremstyle{definition}
\newtheorem{example}[theorem]{Example}
\newtheorem{remark}[theorem]{Remark}

\newtheorem{chunk}[theorem]{}

\usepackage[utf8]{inputenc}

\newtheorem*{ack}{Acknowledgements}

\title[Openness with respect to levels in triangulated categories]{Openness with respect to levels in triangulated categories}

\begin{document}

\begin{abstract}
Given a compactly generated triangulated category $\T$ equipped with an action of a graded-commutative Noetherian ring $R$,
generalizing results of Letz,
we prove a general result concerning the openness with respect to levels of compact objects in $\T$.
Applications are given to derived categories of commutative Noetherian rings, derived categories of commutative Noetherian DG rings and singularity categories.
\end{abstract}

\maketitle
\setcounter{tocdepth}{1}
%\tableofcontents

%\numberwithin{equation}{subsection}

%%%%%%%%%%%%%%%%%%%%%%%%%%%%
\section{Introduction}
\label{sec:introduction}

Let $\mathbb P$ be a property of commutative Noetherian local rings. The $\mathbb P$-locus of a commutative Noetherian ring $R$ is the subset of the spectrum $\Spec(R)$ consisting of prime ideals $\p$ such that the property $\mathbb P$ holds for $R_\p$. It is a classical question whether the $\mathbb P$-locus of a ring $R$ is open in $\Spec(R)$ with respect to the Zariski topology. If $R$ is an excellent ring, the $\mathbb P$-locus of $R$ is open in $\Spec(R)$ for any of the properties $\mathbb P$: regularity, complete intersection, Gorenstein, or Cohen-Macaulay; see \cite{GM1978} and \cite[Section 7]{Grothendieck:1965}.
The openness of the regular locus is of independent interest and has been studied in \cite{Iyengar/Takahashi:2016, Iyengar/Takahashi:2019, Liu:2023, Nagata1959}. Under mild assumptions, it has been more precisely characterized in terms of the cohomological annihilator by Iyengar and Takahashi \cite{Iyengar/Takahashi:2016} and the annihilator of the singularity category by the second author \cite{Liu:2023}.

Let $\mathbb P$ be a property of modules over commutative Noetherian local rings. Similarly, the $\mathbb P$-locus of a module $M$ over a commutative Noetherian ring $R$ is the set consisting of prime ideal $\p$ of $R$ such that the property $\mathbb P$ holds for $M_\p$. The free locus and the finite projective dimension locus of a finitely generated module are always open in $\Spec(R)$; see \cite[Section 9]{Brodmann-Sharp}. 
For an acceptable ring $R$ in the sense of \cite{sharp}, Leuschke \cite{Leuschke:2002} observed that the Gorenstein locus of a finitely generated module is open in $\Spec(R)$. For an excellent ring $R$, Takahashi \cite{Takahashi:2006_Glasgow} proved that the finite injective dimension locus of a finitely generated module is open in $\Spec(R)$. This result was recently extended by Kimura \cite{Kimura} to acceptable rings. Notably, the finite injective dimension locus of the ring itself is precisely the Gorenstein locus of the ring. 

Given a triangulated category $\T$,
a subcategory $\G$ of $\T$, and an object $X \in \T$,
the \emph{level of $X$ with respect to $\G$},
denoted by $\level_\T^{\G}(X)$, 
was first defined in \cite[2.3]{ABIM:2010} as a measurement for the number of steps needed to build $X$ from $\G$.
Its definition is recalled in \ref{level}.

Let $R$ be a commutative Noetherian ring and $X,G$ be objects in the bounded derived category $\D^f_b(R)$.  For each $n\geq 0$, Letz \cite{Letz:2021} observed that the set
   $$
       \{\p\in \Spec(R)\mid \level_{\D^f_b(R_\p)}^{G_\p}(X_\p)\leq n\}
      $$
      is open in $\Spec(R)$.

The main result of this paper is a far reaching generalization of this result. We work in the Benson-Iyengar-Krause \cite{BIK:2008} framework of $R$-linear triangulated categories,
which allows one to perform localization in an arbitrary triangulated category equipped with an action of a graded-commutative Noetherian ring $R$. 
In this setting, we prove the following general open loci result:

\begin{theorem}\label{T2}
(See \ref{main})
Let $R$ be a graded-commutative Noetherian ring, and let $\T$ be an $R$-linear compactly generated triangulated category and $\C$ be a triangulated subcategory of $\T^c$. For each $n\geq 0$, $X\in \C$, and a full subcategory $\G\subseteq \C$, the set
      $$
       \{\p\in \Spec(R)\mid \level_{\C_\p}^{\G}(X)\leq n\}
      $$
      is open in $\Spec(R)$. In particular, the set $\{\p\in \Spec(R)\mid \level^\G_{\C_{\p}}(X)<\infty\}$ is open in $\Spec(R)$.
\end{theorem}

In the above result, $\C_\p$ is the localization of $\C$ at $\p$ (cf. \ref{localization at p}), and $\T^c$ is the full subcategory of $\T$ consisting of compact objects (see  \ref{def of compact}).

Theorem \ref{T2} builds on the work of Letz \cite{Letz:2020, Letz:2021}.  A new ingredient is a model of this result that applies effectively to a triangulated subcategory within the full subcategory of compact objects in a compactly generated triangulated category equipped with an action of a graded-commutative Noetherian ring. Letz's proof uses the argument from the proof of the converse coghost lemma for the bounded derived category, as established by Oppermann and Šťovíček \cite{OS:2012}. In contrast, our proof of Theorem \ref{T2} is based on the converse ghost lemma in $\T$; see Lemma \ref{converse-compact}.

The remainder of the paper is dedicated to applications of this result.
First, we explain in Remark \ref{coincide} how the result of Lets follows from \ref{T2}.

We then show that for a commutative Noetherian ring $R$ and objects $X,G$ in the singularity category $\D_{\sg}(R)$, the set
$
\{\p\in \Spec(R)\mid \level_{\D_{\sg}(R_\p)}^{G_\p}(X_\p)\leq n\}
$
is open in $\Spec(R)$; see Corollary \ref{application-singularity}.
A similar result holds for certain derived categories related to modular representation theory of finite groups, see Example \ref{modular-example}.
Another application of Theorem \ref{T2}
is the local-global principle for $\C$; see Corollary \ref{local-global principle}. This extends Letz's local-global principle \cite{Letz:2021} for the bounded derived category over a commutative Noetherian ring. 

Our next task is to apply \ref{T2} to derived categories of commutative DG rings.
Let $A$ be a non-negative commutative DG ring (where we grade homologically) such that $\h_0(A)$ is Noetherian and $\h(A)$ is finitely generated over $\h_0(A)$.  The third author \cite{Shaul2022} established that the Gorenstein locus $$\Gor_{\h_0(A)}(A)\colonequals\{\bar\p\in \Spec(\h_0(A))\mid A_{\bar \p} \text{ is a Gorenstein DG ring}\}$$
is open in $\Spec(\h_0(A))$ if $A$ has a dualizing DG module in the sense of \cite{Yekutieli}; see Section \ref{notation} for the definition of $A_{\overline{\p}}$ and \ref{Gorenstein def} for Gorenstein DG rings.
Here, much more generally, and using \ref{T2},
we study the finite injective dimension 
 locus of a DG $A$-module $X$, denoted by $\FID_{\h_0(A)}(X)$, over the non-negative commutative DG ring $A$; see \ref{FID}. The main result concerning this is the following. 
\begin{theorem}\label{T1}
(See \ref{FID open}) Let $A$ be a non-negative commutative DG ring such that $A^\natural$ is Noetherian and $\h(A)$ is finitely generated over $\h_0(A)$.   Assume $A$ has a point-wise dualizing DG module. For each $X\in \D^f_b(A)$, 
 $\FID_{\h_0(A)}(X)$ is open in $\Spec(\h_0(A))$.
\end{theorem}
The notion of a point-wise dualizing DG module here is less restrictive than the definitions of a dualizing DG module as given in \cite{FIJ} and \cite{Yekutieli}; see Remark \ref{compare definitions}.
Theorem \ref{T1} recovers the following result due to Kimura \cite{Kimura}: For an essentially finite type algebra over a Gorenstein ring, the finite injective dimension locus of a finitely generated module is open; see Remark \ref{ess finite type}.

The article is organized as follows. In Section \ref{notation}, we recall some basic notations used throughout the article.  In Section \ref{localization-section}, we explore the openness with respect to levels in triangulated categories. In Section \ref{local-global section}, we establish the local-global principle for bounded derived categories over DG rings. In Section \ref{Section: DG of Gabber}, we prove a DG version of a result by Gabber, giving necessary and sufficient conditions for a finitely generated DG module to have finite injective dimension locally. Finally, in Section \ref{openness section}, we investigate the openness of the finite injective dimension locus over DG rings.
% \begin{ques} Let $n\ge 1$ be an integer and $M,N\in D_{sg}(R)$. Is the set $$\{\p\in \Spec(R)\mid M_{\p} \in \thick_{D_{sg}(R_{\p})}^n(N_{\p})\}$$ Zariski open?
% \end{ques}

\begin{ack}
We would like to thank Leonid Positselski for his detailed explanation and helpful discussions of  Lemma \ref{compact}. Souvik Dey was partially supported by the Charles University Research Center program No.UNCE/SCI/022 and a grant GA \v{C}R 23-05148S from the Czech Science Foundation. Jian Liu was supported by the National Natural Science Foundation of China (No. 12401046) and the Fundamental Research Funds for the Central Universities (No. CCNU24JC001).
\end{ack}

\section{Notation, terminology, and preliminary}\label{notation}
A differential graded (DG) ring $A$ is called \emph{commutative} if it is graded-commutative and satisfies $a^2=0$ for any element $a$ of odd degree. 
Note that we grade homologically, so that the differential has degree $-1$.
In this article, we will focus on the study of non-negative commutative DG ring $A$ such that $\h_0(A)$ is Noetherian and $\h(A)$ is finitely generated over $\h_0(A)$; see more details of the DG rings in  \cite{Yekutieli:book}. For each DG $A$-module $M$, set
$$
\sup(M)=\sup\{i\in\Z\mid \h_i(M)\neq 0\}\text{ and } \inf(M)=\inf\{i\in\Z\mid \h_i(M)\neq 0\}.
$$
The \emph{amplitude} of $M$, denoted by $\amp(M)$, is the number $\sup(M)-\inf(M)$.
\begin{chunk}\label{local DG ring}
 \textbf{Homogeneous spectrum and localizations.} For a graded-commutative Noetherian ring $R$, let $\Spec(R)$ denote the set consisting of all homogeneous prime ideals of $R$, and let $\maxSpec(R)$ denote its subset consisting of all 
 maximal homogeneous prime ideals. The set $\Spec(R)$ is endowed with the Zariski topology,  where the closed subsets are of the form $V(I)\colonequals\{\p\in \Spec(R)\mid I\subseteq \p\}$ for some graded ideal $I$ of $R$. 

    Let $A$ be a non-negative commutative DG ring. 
For each $\p\in \Spec(A_0)$, set $A_\p=A\otimes_{A_0} (A_0)_\p$ and $M_\p=M\otimes_{A_0} (A_0)_\p$ for each DG $A$-module $M$. $A_\p$ is still a non-negative commutative DG ring and $M_\p$ is a DG module over $A_\p$. For each prime ideal $\bar\p\in \Spec(\h_0(A))$, $\p= \pi^{-1}({\bar\p})$ is a prime ideal of $A_0$, where $\pi\colon A_0\rightarrow \h_0(A)$ is the canonical map. Set $A_{\bar\p}= A_\p$ and $M_{\bar\p}= M_\p$ for each DG $A$-module $M$.
$A$ is said to be \emph{local} if, in addition, $\h_0(A)$ is a local ring. For example, for each $\bar\p\in \Spec(\h_0(A))$, $A_{\bar\p}$ is a local DG ring.  
\end{chunk}

\begin{chunk}\label{derived category}
 \textbf{Derived categories of DG rings.} For a non-negative commutative DG ring $A$, let $\D(A)$ denote the derived category of DG $A$-modules; see more details in \cite{ABIM:2010} and \cite{Yekutieli:book}. $\D(A)$ is a triangulated category with the suspension functor $\Sigma$; for each $X\in \D(A)$, $(\Sigma X)_i=X_{i-1}$, $a\cdot \Sigma x=(-1)^{|a|}a\cdot x$, $\partial^{\Sigma X}=-\partial^X$.
 
If $\h_0(A)$ is Noetherian, let $\D^f(A)$ denote the full subcategory of $\D(A)$ consisting of DG $A$-modules $X$ such that $\h_i(X)$ is finitely generated over $\h_0(A)$ for each $i\in \mathbb Z$, and let $\D_b(A)$ denote the full subcategory of $\D(A)$ consisting of DG $A$-modules $X\in \D(A)$ such that $\h_i(X)=0$ for $|i|\gg 0$. Set
$$
\D^f_b(A)=\D^f(A)\cap \D_b(A).
$$
If $\h_0(A)$ is Noetherian and $\h(A)$ is finitely generated over $\h_0(A)$, then $\D^f_b(A)$ is precisely the full subcategory of $\D(A)$ consisting of DG $A$-modules $X$ such that the total homology $\h(X)$ is finitely generated over $\h_0(A)$.
\end{chunk}

\begin{chunk}
    \textbf{Thick subcategories.} For a triangulated category $\T$, a full subcategory of $\T$ is called \emph{thick} if it is closed under suspensions, cones, and direct summands. Let $\G$ be a full subcategory of $\T$, the smallest thick subcategory of $\T$ containing $\G$ will be denoted by $\thick_\T(\G)$. Following \cite[2.2.4]{ABIM:2010} and \cite[Section 2]{Bondal/vdB:2003}, it can be inductively constructed as below. 
    
Set $\thick_{\T}^0(X)=\{0\}$. Denote by $\thick_\T^1(\G)$ the smallest full subcategory of $\T$ that contains $\G$ and is closed under finite direct sums, direct summands, and suspensions. Inductively, $\thick_\T^n(\G)$ is denoted to be the full subcategory of $\T$ consisting of objects $X\in\T$ that appear in an exact triangle
$$
X_1\rightarrow X\oplus X^\prime\rightarrow X_2\rightarrow \Sigma X_1,
$$
where $X_1\in \thick^1_{\T}(\G)$ and $X_2\in \thick_{\T}^{n-1}(\G)$.
The smallest thick subcategory of $\T$ containing $X$ is precisely equal to $\displaystyle\bigcup_{n\geq 0}\thick_\T^n(\G)$.
\end{chunk}

\begin{chunk}\label{def of compact}
   \textbf{Compactly generated triangulated categories.} Let $\T$ be a triangulated
category with arbitrary direct sums. An object $X \in \T$ is called \emph{compact} if, for each class of objects
$X_i(i \in I) \in \T$ , the canonical map
$$
\bigoplus_{i\in I}
\Hom_\T (X, X_i) \longrightarrow  \Hom_\T
(X,\bigoplus_{
i\in I}
X_i)$$
is an isomorphism. We write $\T^c$ to be the full subcategory of $\T$ consisting of all compact
objects in $\T$. Note that $\T^c$ 
is a thick subcategory of $\T$.

The category $\T$ is said to be \emph{compactly generated} provided that there exists a set $S$ consisting of compact
objects in $\T$ such that any object Y satisfying $\Hom_\T(X, \Sigma^i
Y) = 0$ for all $X \in S$ and
$i \in \mathbb Z$ is zero. In this case, $\T^c =
\thick_\T (S)$; see \cite[Lemma 2.2]{Neeman-norm} and \cite[Lemma 8.3.1]{Neeman-book}. For example, if $R$ is a ring, $\D(R)$ is compactly generated by the compact object $R$, and hence $\D(R)^c=\thick_{\D(R)}(R)$.
\end{chunk}

\begin{chunk}\label{level}
    \textbf{Levels in triangulated categories.} Let $\T$ be a triangulated category. For a full subcategory $\G$ of $\T$ and $X\in \T$, following \cite[2.3]{ABIM:2010}, the \emph{level of $X$ with respect to $\G$} is defined to be 
    $$
\level_\T^{\G}(X)\colonequals\inf\{n\in \mathbb Z\mid X\in \thick_\T^n(\G)\}.
$$
Note that $\level_\T^\G(X)<\infty$ if and only if $X\in \thick_\T(\G)$. If there is an exact functor $F\colon \T\rightarrow \T^\prime$ between triangulated categories, then $\level_{\T^\prime}^{F(\G)}(F(X))\leq \level_\T^{\G}(X)$; see \cite[Lemma 2.4]{ABIM:2010}.
\end{chunk}
\begin{chunk}\label{def of ghost index}
    \textbf{Ghost map and ghost index.} Let $\T$ be a triangulated category, and let $\G$ be a full subcategory of $\T$. A morphism 
$f: X \rightarrow Y$ in $\T$ is called $\G$-\emph{ghost} if 
$$
\Hom_T(\Sigma^i G, f): \Hom_T(\Sigma^i G, X) \to \Hom_T(\Sigma^i G,Y)
$$
is zero for each $i \in \mathbb{Z}$ and $G\in G$. 
For example, if $R$ is a ring, $\T = \D(R)$ and $\G=\{R\}$,
a $\G$-ghost is a morphism $f$ such that $\h_n(f) = 0$ for all $n$.
Following \cite[Section 2]{Letz:2021}, the \emph{ghost index of $X$ 
with respect to $\G$} in $T$, denoted by $\gin^\G_\T(X)$, to be the smallest non-negative integer 
$n$ such that any composition of $n$-fold $\G$-ghost maps 
$$
X \xrightarrow{f_0} X_{1} \xrightarrow{f_{1}} \cdots \xrightarrow{f_{n-1}} X_n
$$
which starts with $X$ is zero in $\T$.
\end{chunk}

    Recall that a full subcategory $\C$ of an additive category $\T$ is called \emph{contravariantly finite} if, for each $X\in \T$, there exist $C\in \C$ and a morphism $\pi\colon C\rightarrow X$ in $\T$ such that any morphism $g\colon 
 C^\prime \rightarrow X$ with $C^\prime\in C$ factors through $\pi$ (i.e., there exists a morphism $g^\prime\colon C^\prime\rightarrow C$ such that $g=\pi\circ g^\prime)$.

\begin{chunk}\label{contravariantly finite}
    Let $\T$ be a triangulated category with arbitrary direct sums, and let $\G$ be a full subcategory of $\T$. 
    
(1) Let $f\colon X\rightarrow Y$ be a morphism in $\T$. Embed $f$ into an exact triangle
       $
       Z\xrightarrow g X\xrightarrow f Y\rightarrow \Sigma Z.
       $
       If $f$ is a $\G$-ghost, then every morphism $G\rightarrow X$ with $G\in \G$ factors through $g$. The converse holds if, in addition, $\G$ is closed under suspensions.

(2) Denote by $\Add_{\T}(\G)$ the smallest full subcategory of $\T$ that contains $\G$ and is closed under direct sums and suspensions. Note that $\Add_\T(\G)$ is contravariantly finite in $\T$; see \cite[Remark 2.5]{Beligiannis:2008} for the case where $\G$ consists of a single object, with the same reason applied to a subcategory. Thus, for each $X\in \T$, there exist $C\in \Add_\T(\G)$ and a morphism $\pi\colon C\rightarrow X$ in $\T$ such that any morphism $G \rightarrow X$ with $G\in \G$ factors through $\pi$. Note that a morphism in $\T$ is $\G$-ghost if and only if it is $\Add_{\T}(\G)$-ghost. Combining this with (1), for each $X\in \T$, there exists an exact triangle
$$
C\xrightarrow\pi X\xrightarrow f X_1\rightarrow \Sigma C,
$$
where $C\in \Add_\T(\G)$ and $f$ is a $\G$-ghost map.

\end{chunk}
\begin{chunk}\label{converse}
\textbf{Converse ghost lemma.} Let $\T$ be a triangulated category with arbitrary direct sums, and let $\G$ be a full subcategory of $\T$ that is closed under suspensions. The ghost lemma is an inequality between the ghost number and the level, specifically $\gin_\T^\G(X)\leq \level^\G_\T(X)$; see \cite[Lemma 2.2]{Beligiannis:2008}.  If, in addition, $\G$ is contravariantly finite, then 
    $$
    \gin^\G_\T(X)=\level^\G_\T(X)
    $$
    for each $X\in \T$; see \cite[Lemma 2.2]{Beligiannis:2008}. We refer to this equality as the converse ghost lemma for $\T$. 
    
  Whether the ghost index and the level coincide in an arbitrary triangulated category $\T$ is still an open question. 
   Dually to the ghost map and the ghost index, there are also notions of the coghost map and the coghost index; see \cite[Section 2]{Letz:2021}. There are also corresponding versions of the converse coghost lemma. 
\end{chunk}

\section{Openness with respect to levels in triangulated categories}\label{localization-section}
Throughout this section, let $R$ be a graded-commutative Noetherian ring,  and let $\mathcal{T}$ be an $R$-linear triangulated category; see \ref{R-linear}. We investigate the property of openness with respect to levels in a triangulated subcategory $\C$ of $\T$. It turns out that this property is well-behaved when $\T$ is compactly generated and $\C\subseteq \T^c$; see Theorem \ref{main}.  As mentioned in the introduction, our result builds on ideas from Letz’s work \cite[Proposition 3.5]{Letz:2021} regarding bounded derived categories and extends her result to our context; see Remark \ref{coincide}.

\begin{chunk}\label{R-linear}
  For objects $X,Y$ in the triangulated category $\T$, set $\Hom_{\mathcal T}^\ast(X,Y)\colonequals\displaystyle\bigoplus_{i\in \mathbb Z}\Hom_\T(X,\Sigma^i Y)$.  Recall that $\T$ is said to be $R$-linear if, for each object $X \in T$, there is a 
homomorphism
$$\varphi_X\colon R \rightarrow \Hom_\T^\ast(X,X)$$
of graded rings
such that the right $R$-action on $\Hom_\T^\ast (X,Y )$ via $\varphi_X$ and the left $R$-action via $\varphi_Y$ are
compatible for $X,Y\in\T$. Specifically, for homogeneous elements $r \in R$ and  $f\in \Hom_\T^\ast (X,Y )$, one has
$$\varphi_Y (r) \circ f = (-1^{|r||f|}f\circ\varphi_X(r).$$
It follows that $\Hom_\T^\ast(X,Y)$ has a graded $R$-module structure with $r\cdot f=\varphi_Y(r)\circ f$.
\end{chunk}
\begin{chunk}\label{localization at p}
    Let $V$ be a specialization closed subset of $\Spec(R)$. That is, if $\p\in V$, then $\q\in V$ if $\p\subseteq \q$. For each triangulated subcategory $\mathcal{C}$ of $\mathcal{T}$, following Benson, Iyengar, and Krause \cite[Section 4]{BIK:2008}, the full subcategory $\mathcal{C}_V$ of $\C$ is defined to be
$$
\mathcal{C}_V\colonequals \{X\in \mathcal{C}\mid \Hom_{\mathcal{T}}^\ast(C,X)_\p=0 \text{ for all }C\in \mathcal{T}^c, ~\p\in \Spec(R)\setminus V\};
$$
see also \cite[Section 2]{BIK:2011}. Note that $\mathcal{C}_V$ is a thick subcategory of $\mathcal{C}$. 

For each $\p$ in $\Spec(R)$, set
$
Z(\p)\colonequals \{\q\in \Spec(R)\mid \q \nsubseteq \p\}.
$
The set $Z(\p)$ is a specialization closed subset of $\Spec(R)$. The \emph{localization} of $\mathcal{C}$ at $\p$ is defined to be
$$
\mathcal{C}_\p\colonequals \mathcal{C}/\mathcal{C}_{Z(\p)}.
$$
\end{chunk}
%For each $X\in \C$, we write $X_\p$ to be the image of $X$ in $\C/\C_{Z(\p)}$.

\begin{chunk}\label{compute C_V}
    Assume $\mathcal{C}\subseteq \mathcal{T}^c$. According to the definition, 
    $$
\mathcal{C}_V= \{X\in \mathcal{C}\mid \Hom_{\mathcal{T}}^\ast(X,X)_\p=0 \text{ for all } \p\in \Spec(R)\setminus V\}
$$
  In this case,  
    the definition of $\mathcal{C}_V$ coincides with that in \cite[Section 3]{BIK:2015}, and hence the definition of $\mathcal{C}_\p$ is also the same.
\end{chunk}

\begin{example}\label{Krause's observation}
Assume $R$ is a commutative Noetherian ring. Let $\mathcal{T}=\K(\Inj R)$ denote the homotopy category of complexes of injective $R$-modules. This is $R$-linear. By taking injective resolutions, Krause \cite[Proposition 2.3]{Krause:2005} observed that there is a triangle equivalence $\D^f_b(R)\cong \K(\Inj R)^c.$ 

 Let $\p$ be a prime ideal of $R$. Since $R$ is Noetherian,  
$
\Hom_{\D^f_b(R)}(X,Y)_\p\cong \Hom_{\D^f_b(R_\p)}(X_\p,Y_\p)
$
for $X,Y\in \D^f_b(R)$.
It follows from this and \ref{compute C_V} that
$
\D^f_b(R)_{Z(\p)}=\{X\in \D^f_b(R)\mid X_\p =0 \text{ in } \D^f_b(R_\p)\}.
$
Matsui {\rm\cite[Lemma 3.2 (2)]{Matsui:2021}} observed that $\D^f_b(R)/\D^f_b(R)_{Z(\p)}\cong \D^f_b(R_\p)$, and hence $\D^f_b(R)_\p\cong \D^f_b(R_\p).$ We will prove the DG version of this result in Proposition \ref{D^f(A)localization}.
\end{example}
For each $X\in\T$ and each homogeneous element $r\in R$, the \emph{Koszul object} of $r$ on $X$,
denoted by $X \para r$, is the object that fits into the exact triangle
$
X\xrightarrow{r} \Sigma^{|r|}X\rightarrow X\para r\rightarrow \Sigma X;
$ 
here, $r\colon X\rightarrow \Sigma^{|r|}X$ is represented by $\varphi_X(r)$ as in \ref{R-linear}. The following result follows from \cite[Lemma 3.5]{BIK:2015}; see also \cite[Lemma 3.1]{Matsui:2021}.

\begin{lemma}\label{iso}
  Let $\mathcal{C}$ be a triangulated subcategory of $\mathcal{T}^c$. For each $\p\in \Spec(R)$,

  \begin{enumerate}
      \item $\mathcal{C}_{Z(\p)}=\thick_\mathcal{T}(X\para r\mid X\in \mathcal{C}, r\notin \p\}$.

      \item The quotient functor $\mathcal{C}\rightarrow \mathcal{C}/\mathcal{C}_{Z(\p)}=\mathcal{C}_\p$ inudces a natural isomorphism
      $$
\Hom_\mathcal{C}^\ast(X,Y)_\p\xrightarrow\cong \Hom_{\mathcal{C}_\p}^\ast(X,Y)
$$
for each $X,Y\in \mathcal{C}$.
  \end{enumerate}
\end{lemma}

\begin{chunk}\label{localization map}
    For  each $X \in \T$ and $r\notin \p$, it is routine to prove that $X\para r\in \T_{Z(\p)}$; see the proof of \cite[Lemma 3.1 (b)]{Matsui:2021}. Hence, for $X,Y\in \T$, there is a natural map
$$
    \Phi\colon \Hom_\T^\ast(X,Y)_\p\rightarrow \Hom_{\T_\p}^\ast(X,Y);~~ \frac{f}{r}\mapsto f/r\colonequals (\Sigma^{|r|}X\xleftarrow r X\xrightarrow f \Sigma^{|f|}Y),
$$
where $f/r$ is a right fraction. 
\end{chunk}

\begin{lemma}\label{bijective}
    If $X\in \mathcal T^c$, then the map $\Phi$ in \ref{localization map} is an isomorphism.
\end{lemma}
\begin{proof}
    Assume $\Phi(\frac{f}{r})=0$. Then there exists a map $\alpha\colon Z\rightarrow X$ in $\T$ such that $\rm{cone}(\alpha)\in \T_{Z(\p)}$ and $f\circ \alpha=0$. Consider the exact triangle
    $$
Z\xrightarrow \alpha X\xrightarrow \beta {\rm{cone}}(\alpha) \rightarrow \Sigma Z.
$$
Since $f\circ \alpha=0$, $f$ factors through $\beta$. Combining $\rm{cone}(\alpha)\in \T_{Z(\p)}$ and $X\in \T^c$, it follows from the definition of $\T_{Z(\p)}$ that $\frac{\beta}{1}=0$ in $\Hom_\T^\ast(X,{\rm cone}(\alpha))_\p$. Thus, there exists a homogeneous element $s\notin \p$ such that $s\cdot \beta=0$. Since $f$ factors through $\beta$, we get that $s\cdot f=0$. In particular, $\frac{f}{r}=0$, and hence $\Phi$ is injective.

Let $X\xleftarrow t M\xrightarrow g \Sigma^i Y$ be a morphism in $\Hom_{\T_\p}^\ast(X,Y)$. Consider the exact triangle
$$
M\xrightarrow t X\xrightarrow \gamma {\rm cone}(t) \rightarrow \Sigma M.
$$
By the choice of $t$, ${\rm cone}(t)\in \T_{Z(\p)}$. With the same argument of the existence of $s\notin\p$ as above, there exists a homogeneous element $a\notin \p$ such that $a\cdot \gamma=0$. It follows that $\gamma \circ (a\cdot  {\rm id}_X)=0$; see \ref{R-linear}. Combining with this, we conclude from the exact triangle that there exists $\eta\colon \Sigma^{-|a|}X\rightarrow M$ such that $t\circ \eta=a\cdot {\rm id}_X$. This implies $\Phi(\frac{g\circ \eta}{a})=(g\circ \eta)/ (t\circ \eta)=g/t$, and hence $\Phi$ is surjective. This completes the proof. 
\end{proof}
 \begin{chunk}\label{cofinal}
Let $\C$ be a triangulated category, and let $\C^\prime, \mathcal D$ be triangulated subcategories of $\C$. Denote $\mathcal D^\prime=\C^\prime\cap \mathcal D$. Suppose that each morphism $D\rightarrow C^\prime$ with $D\in \mathcal D$ and $C^\prime\in \C^\prime$ factors through an object in $\mathcal D^\prime$. Then the induced exact functor 
$$
\C^\prime/\mathcal D^\prime\longrightarrow \C/\mathcal D
$$ 
is fully faithful; see \cite[Chapter 2]{Verdier1977}.
\end{chunk}
\begin{chunk}\label{admitdirectsum}
    Let $\C$ be a triangulated category with arbitrary direct sums. A triangulated subcategory of $\C$ is called \emph{localizing} if it is closed under direct sums. 
    
    If $\mathcal D$ is a localizing subcategory of $\T$, then $\C/\mathcal D$ has arbitrary direct sums and the quotient functor $\C\rightarrow \C/\mathcal D$ preserves direct sums; see \cite[Corollary 3.2.11]{Neeman-book}. 
   Consequently, for an $R$-linear triangulated triangulated category $\T$ with arbitrary direct sums, $\T_\p$ has arbitrary direct sums as $\T_{Z(\p)}$ is a localizing subcategory of $\T$. 
\end{chunk}
\begin{proposition}\label{embedding}
    Let $\T$ be an $R$-linear triangulated category and $\p$ be a prime ideal of $R$. 
    \begin{enumerate}
        \item If $\C$ is a triangulated subcategory of $ \T^c$, then
 the inclusion $\mathcal{C}\rightarrow \T$ induces a fully faithful functor $$\C_\p\rightarrow \T_\p.$$

 \item Assume $\T$ is compactly generated. Then both $\T_{Z(\p)}$ and $\T_\p$ are compactly generated.  Moreover, the quotient functor  $\T\rightarrow \T_\p$ preserves compact objects and it induces a triangle equivalence (up to direct summands)
    $$
   (\T^c)_\p\xrightarrow \cong (\T_\p)^c.
    $$
    \end{enumerate}
\end{proposition}    
\begin{proof}
  (1) For each $X,Y\in \C$, the statement follows immediately from the commutative diagram
    $$
    \xymatrix{
\Hom_{\C_\p}(X,Y)_\p\ar[r]^-\cong\ar@{=}[d]& \Hom_{\C_\p}(X,Y)\ar[d]\\
\Hom_{\T}(X,Y)_\p\ar[r]^-\cong & \Hom_{\T_\p}(X,Y),
}
    $$
where the above two isomorphisms follow from Lemma \ref{iso} and Lemma \ref{bijective} respectively.

(2)  By \cite[Proposition 2.7]{BIK:2011},   $\T_{Z(\p)}$ is compactly generated by the set $\{X\para r\mid X\in \T^c, r\notin\p\}$. Combining with Lemma \ref{iso}, we conclude that $(\T^{c})_{Z(\p)}=(\T_{Z(\p)})^c$. 
Since $X\para r\in \T^c$ for each $X\in \T^c $, $(\T_{Z(\p)})^c\subseteq \T^c$.
By \cite[Theorem 2.1]{Neeman-norm},  $\T_\p=\T/\T_{Z(\p)}$ is compactly generated, the quotient functor  $\T\rightarrow \T/\T_{Z(\p)}$ preserves compact objects, and the quotient functor $\T\rightarrow \T/\T_{Z(\p)}$ induces a triangle equivalence (up to direct summands)
    $$
    \T^c/(\T_{Z(\p)})^c\xrightarrow\cong (\T/\T_{Z(\p)})^c.
    $$
Note that $\T^c/(\T_{Z(\p)})^c=(\T^c)_\p$ as $(\T^c)_{Z(\p)}=(\T_{Z(\p)})^c$. This completes the proof.
\end{proof}

\begin{chunk}\label{Adams}
 Let $\T$ be an $R$-linear triangulated category with arbitrary direct sums. For each $X\in \T$ and a full subcategory $\G$ of $\T$, it follows from \ref{contravariantly finite} that there exists a sequence of morphisms 
$$
X=X_0\xrightarrow {f_0} X_1\xrightarrow {f_1}X_2 \xrightarrow {f_2}\cdots
$$
in $\T$ such that, for each $i\geq 0$, $f_i$ is a $\G$-ghost and the cone of $f_i$ belongs to $\Add_{\T}(\G)$. 
%there is an exact triangle
%$$
%Y_i\xrightarrow {\pi_i} X_i\xrightarrow{f_{i}} X_{i+1}\rightarrow \Sigma Y_i
%$$
%in $\T$ with $\pi_i$ being a right $\Add_{\T}(G)$-approximation. 
Such a sequence of $\G$-ghost morphisms as above is called an \emph{Adams resolution} of $X$; see \cite[Section 4]{Christensen:1998}.
\end{chunk}
\begin{chunk}\label{gin}
  Using the notation as \ref{Adams}, assume that the sequence  $
X\xrightarrow {f_0} X_1\xrightarrow {f_1}X_2 \xrightarrow {f_2}\cdots
$ is an Adams resolution of $X$. By \cite[Lemma 3.2.1 and 3.4.3]{Letz:2020},
$$
\gin^\G_\T(X)={\rm inf}\{n\geq 0\mid f_{n-1}\circ f_{n-2}\circ \cdots \circ f_0=0\};
$$
see the definition of the ghost index $\gin_\T^\G(X)$ in \ref{def of ghost index}.
 Indeed, this follows because, for any sequence of $\G$-ghost morphisms $
X\xrightarrow {g_0} Y_1\xrightarrow {g_1}Y_2 \xrightarrow {g_2}\cdots
$ in $\T$, there exists a commutative diagram
$$
\xymatrix{
X\ar@{=}[d]\ar[r]^-{f_0}& X_1\ar[r]^-{f_1}\ar@{-->}[d]^-\exists & X_2\ar[r]^-{f_2}\ar@{-->}[d]^-\exists & \cdots \\
X\ar[r]^-{g_0}& Y_1\ar[r]^-{g_1}& Y_2\ar[r]^-{g_2}& \cdots.
}
$$
The existence of vertical maps is guaranteed as the cone of each $f_i$ is in $\Add_\T(\G)$ and each $g_i$ is a $\G$-ghost.
\end{chunk}
\begin{lemma}\label{localAdams}
Let $
X\xrightarrow {f_0} X_1\xrightarrow {f_1}X_2 \xrightarrow {f_2}\cdots$
be an Adams resolution of $X$ as \ref{Adams}, and assume that $\G\subseteq  \T^c$. For each $\p\in \Spec(R)$ and $i\geq 1$, the map $f_i/1\colon X_i\rightarrow X_{i+1}$ is a $\G$-ghost map in $\T_\p$, and the cone of $f_i/1$ is in $\Add_{\T_\p}(\G)$. In particular, the sequence $$
X\xrightarrow {f_0/1} X_1\xrightarrow {f_1/1}X_2\xrightarrow {f_2/1} \cdots
$$
is an Adams resolution of $X$ in $\T_\p$.
\end{lemma}
\begin{proof}
   Since $\G\subseteq \T^c$, it follows from Lemma \ref{bijective} that $f_i/1$ is a $\G$-ghost map in $\T_\p$ for each $i\geq 0$. Let $\Cone(f_i)$ denote the cone of $f_i$ in $\T$. By definition of the Adams resolution, we have $\Cone(f_i)\in \Add_\T(\G)$. Noting that $\T_{Z(\p)}$ is a localizing subcategory of $\T$, we apply \ref{admitdirectsum} to conclude that $\T_\p$ has arbitrary direct sums, and that the quotient functor $\T\rightarrow \T_\p$ preserves direct sums. Thus, $\Cone(f_i)\in \Add_{\T_\p}(\G)$. In $\T_\p$, $\Cone(f_i/1)\cong \Cone(f_i)$, and hence $\Cone(f_i/1)\in \Add_{\T_\p}(\G)$.
\end{proof}
Let $\mathcal{T}_1$ be a thick subcategory of a triangulated category $\T_2$. For each $X\in \mathcal{T}_1$ and any subcategory $\G$ of $\mathcal T_1$, there is an equality $\level_{\T_1}^{\mathcal{G}}(X)=\level^{\mathcal G}_{\T_2}(X)$; see \cite[Lemma 2.4]{ABIM:2010}. Proposition \ref{equality} extends this to any fully faithful embedding of triangulated categories. Before this, we need the following lemma. It should be noted that the essential image of a fully faithful functor $\T_1\rightarrow T_2$ between the triangulated categories is not necessarily thick in $\T_2$. 
\begin{lemma}\label{basic}
    Let $F\colon \mathcal T_1\rightarrow \T_2$ be a fully faithful functor between triangulated categories. 
    \begin{enumerate}
        \item For a full subcategory $\G$ of $\mathcal{T}_1$ and $n\geq 0$, if $X\in \thick^n_{\T_2}(F(\G))$, then there exists $X^\prime \in \thick_{\T_1}^n(\G)$ such that $X$ is a direct summand of $F(X^\prime)$ in $\T_2$.

        \item For $X,Y\in \mathcal C$, if $F(X)$ is a direct summand of $F(Y)$ in $\T_2$, then $X$ is a direct summand of $Y$ in $\mathcal{T}_1$.
    \end{enumerate}
\end{lemma}
\begin{proof}
    (1)  We prove the statement by induction on $n$. The case when $n=0$ is trivial. For the case $n=1$, the assumption implies that $X$ is a direct summand of a finite direct sum of objects of the form $\Sigma^iF(G)^{\beta_i}$, where $G\in \G$ and $\beta_i\geq 0$. The desired result follows as $\Sigma^iF(G)^{\beta_i}\cong F(\Sigma^iG^{\beta_i})$.
    
    %For $n=1$, assume $X\in \thick^1_{\T_2}(F(\G))$. By definition of $\thick^1_{\T_2}(F(\G))$,
  %  $X$ is a direct summand of $(\Sigma^{l_1}F(C_1))^{\beta_1}\oplus \cdots \oplus(\Sigma^{l_s}F(C_s))^{\beta_s}$
%    for some $C_1,\ldots,
%C_s\in \G$, $l_i\in \mathbb Z$, and $\beta_i\geq 0$. Choose $X$ to be $(\Sigma^{l_1}C_1)^{\beta_1}\oplus \cdots\oplus (\Sigma^{l_s}C_s)^{\beta_s}$. Then $X^\prime\in \thick^1_{\T_1}(\G)$, and $X$ is a direct summand of $F(X^\prime)$.
      
Assume $n>1$ and the desired result holds if $\level^{F(\G)}_{\T_2}(X)\leq n-1$. Now, suppose $\level_{\T_2}^{F(\G)}(X)=n$. Then there exists an exact triangle
  $
M\rightarrow Y\rightarrow N\rightarrow \Sigma M
$
in ${\T_2}$, where $M\in \thick^{n-1}_{\T_2}(F(\G))$, $N\in \thick^1_{\T_2}(F(\G))$, and $X$ is a direct summand of $Y$. By induction, there exists $M^\prime \in \thick_{\mathcal {T}_1}^{n-1}(\G)$ and $N^\prime \in \thick^1_{\mathcal T_1}(\G)$ such that $M$ is a direct summand of $F(M^\prime)$ and $N$ is a direct summand of $F(N^\prime)$. By taking direct sums with exact triangles like $X_1\xrightarrow {\id} X_1\rightarrow 0\rightarrow \Sigma X_1$ and $0\rightarrow X_2\xrightarrow {\id} X_2\rightarrow 0$ for some $X_1,X_2\in \T_2$, there exists an exact triangle
\begin{equation}\label{triangle}
    F(M^\prime)\rightarrow Z\rightarrow F(N^\prime)\xrightarrow \alpha \Sigma F(M^\prime)
\end{equation}
and $Y$ is a direct summand of $Z$. It follows that $X$ is a direct summand of $Z$. Since $F$ is fully faithful, there exists $\beta\colon N^\prime\rightarrow \Sigma M^\prime$ such that $F(\beta)$ is the composition $F(X^\prime)\xrightarrow {\alpha} \Sigma F(M^\prime)\cong \Sigma F(M^\prime)$. Consider the exact triangle 
$
M^\prime\rightarrow W\rightarrow N^\prime\xrightarrow \beta \Sigma M^\prime
$
in $\T_1$. By comparing this exact triangle with (\ref{triangle}), we have $F(W)\cong Z$. Note that $W\in \thick^n_{\T_1}(\G)$ and $X$ is a direct summand of $Z\cong F(W)$. 

    (2) By assumption, there is a split exact triangle
    \begin{equation}\label{splittriangle}
        F(X)\xrightarrow i F(Y)\xrightarrow \pi M\xrightarrow 0 \Sigma F(X)
    \end{equation}
    in ${\T_2}$. 
    Since $F$ is fully faithful, there exists $i^\prime\colon X\rightarrow Y$ such that $F(i^\prime)=i$. Comparing the exact triangle
    \begin{equation}\label{origin}
     X\xrightarrow {i^\prime} Y \xrightarrow {\pi^\prime}{\rm cone}(i^\prime)\xrightarrow \alpha \Sigma X  
    \end{equation}
with (\ref{splittriangle}), we conclude that $F(\alpha)=0$. Hence, $\alpha=0$ as $F$ is fully faithful. This implies that the exact triangle (\ref{origin}) is split. In particular, $X$ is a direct summand of $Y$ in $\T_1$.
\end{proof}

\begin{proposition}\label{equality}
    Let $F\colon \T_1\rightarrow {\T_2}$ be a fully faithful functor between triangulated categories. For each $X\in \T_1$ and a full subcategory $\G\subseteq \T_1$, 
    $$
    \level^\G_{\T_1}(X)=\level^{F(\G)}_{{\T_2}}(F(X)).
    $$
\end{proposition}
\begin{proof}
  By \ref{level},  $\level^\G_{\T_1}(X)\geq \level^{F(\G)}_{{\T_2}}(F(X))$. It remains to prove the converse $\level^\G_{\T_1}(X)\leq \level^{F(\G)}_{{\T_2}}(F(X))$. We may assume that $\level_{\T_2}^{F(\G)}(F(X))=n$ is finite. Hence, $F(X)\in \thick^n_{\T_2}(F(\G))$. 
By Lemma \ref{basic}, $F(X)$ is a direct summand of $F(X^\prime)$ for some $X^\prime\in \thick^n_{\T_1}(\G)$. Again by Lemma \ref{basic}, $X$ is a direct summand of $X^\prime$ in $\T_1$. We conclude that  $\level^\G_{\T_1}(X)\leq n$. 
\end{proof}

The following version of the converse ghost lemma will be used in the proof of Proposition \ref{small-big}.
\begin{lemma}\label{converse-compact}
     Let $\T$ be a triangulated category with arbitrary direct sums.  For each $X\in \T^c$ and a full subcategory $\G\subseteq \T^c$, 
    $$
    \gin^\G_\T(X)=\level^\G_\T(X).
   $$ 
\end{lemma}
\begin{proof}
    Note that a morphism in $\T$ is $\G$-ghost if and only if it is $\Add_{\T}(\G)$-ghost. Hence, we have $\gin^\G_\T(X)=\gin^{\Add_{\T}(\G)}_\T(X)$. Combining with \ref{contravariantly finite} and \ref{converse}, $\gin^{\Add_{\T}(\G)}_\T(X)=\level_\T^{\Add_\T(\G)}(X)$. Thus, $\gin_\T^\G(X)=\level_\T^{\Add_\T(\G)}(X)$. 
Since $G\subseteq \T^c$, it follows from \cite[Proposition 2.2.4]{Bondal/vdB:2003} that $$\thick^n_\T(\Add_{\T}(
    \G))\cap \T^c=\thick^n_\T(
    \G).$$ Combining with $X\in \T^c$, this implies $\level^\G_\T(X)=\level^{\Add_{\T}(\G)}_\T(X)$. This completes the proof.
\end{proof}

\begin{proposition}\label{small-big}
  Let $\T$ be an $R$-linear triangulated category with arbitrary direct sums, and let $\C$ be a triangulated subcategory in $\T^c$. Assume $
X\xrightarrow {f_0} X_1\xrightarrow {f_1}X_2 \xrightarrow {f_2}\cdots$
be an Adams resolution of $X$ as \ref{Adams}. If $X$ is an objective in $\C$ and $\G$ is a full subcategory of $\C$, then:
\begin{enumerate}
    \item $\gin^\G_\T(X)=\level^\G_\T(X)=\level^\G_\C(X)={\rm inf}\{n\geq 0\mid f_{n-1}\circ f_{n-2}\circ \cdots \circ f_0=0\}$.
    \item Assume $\T$ is compactly generated. For each $\p\in \Spec(R)$, then
    $$
    \gin^{\G}_{\T_\p}(X)=\level_{\T_\p}^{\G}(X)=\level_{\C_\p}^{\G}(X)={\rm inf}\{n\geq 0\mid (f_{n-1}\circ f_{n-2}\circ \cdots \circ f_0)/1=0\}.
    $$

\item Assume $\T$ is compactly generated. There are equalities
\begin{align*}
    \level^\G_\C(X)& =\sup\{\level^{\G}_{\C_\p}(X)\mid \p\in \Spec(R)\}\\
   &= \sup\{\level^{\G}_{\C_\m}(X)\mid \m\in \maxSpec(R)\}.
\end{align*}
\end{enumerate}  

\end{proposition}

\begin{proof}
  (1) By Proposition \ref{equality}, $\level^\G_\T(X)=\level^\G_\C(X)$. Combining this with Lemma \ref{converse-compact}, it remains to see $\gin^\G_\T(X)={\rm inf}\{n\geq 0\mid f_{n-1}\circ f_{n-2}\circ \cdots \circ f_0=0\}$. This follows from \ref{gin}.

  (2) Using the map $\Phi$ in \ref{localization map}, we observe that $\T_\p$ is an $R_\p$-linear triangulated category. By Proposition \ref{embedding}, $\T_\p$ is compactly generated, and we can regard $\C_\p$ as a triangulated subcategory of $(\T_\p)^c$. Note that $f_n /1\circ f_{n-1}/1\circ \cdots \circ f_1/1=(f_{n-1}\circ f_{n-2}\circ \cdots \circ f_0)/1$. The desired result now follows by combining Lemma \ref{localAdams} and (1).

  (3) Note that 
  \begin{align*}
    \level^\G_\C(X)& \geq\sup\{\level^{\G}_{\C_\p}(X)\mid \p\in \Spec(R)\}\\
   &\geq \sup\{\level^{\G}_{\C_\m}(X)\mid \m\in \maxSpec(R)\}.
\end{align*} It remains to prove $\level^\G_\C(X)\leq n\colonequals\sup  \{\level^{\G}_{\C_\m}(X)\mid \m\in \maxSpec(R)\}$.  We may assume $n$ is finite. We write $f$ to be $f_{n-1}\circ f_{n-2}\circ \cdots \circ f_0\colon X\rightarrow X_n$. By (2), $f/1\in \Hom_{\T_\m}(X,X_n)$ is $0$ for each $\m\in \maxSpec(R)$. Since $X\in \T^c$, it follows from Lemma \ref{bijective} that $\frac{f}{1}\in \Hom_\T(X,X_n)_\m$ is $0$ for each $\m\in \maxSpec(R)$. This implies that $f\in \Hom_\T(X,X_n)$ is $0$. Then (1) yields that $\level^\G_\C(X)\leq n$.
\end{proof}

\begin{theorem}\label{main}
Let $R$ be a graded-commutative Noetherian ring, and let $\T$ be an $R$-linear compactly generated triangulated category and $\C$ be a triangulated subcategory of $\T^c$. For each $n\geq 0$, $X\in \C$, and a full subcategory $\G\subseteq \C$, the set
      $$
       \{\p\in \Spec(R)\mid \level_{\C_\p}^{\G}(X)\leq n\}
      $$
      is open in $\Spec(R)$. In particular, the set $\{\p\in \Spec(R)\mid \level^\G_{\C_{\p}}(X)<\infty\}$ is open in $\Spec(R)$.
\end{theorem}
\begin{proof}
Set $V_n=\{\p\in \Spec(R)\mid \level_{\C_\p}^{\G}(X)\leq n\}$. The second statement follows from the first one as $\{\p\in \Spec(R)\mid \level_{\mathcal{C}}^\G(X)<\infty)\}=\bigcup_{n\geq 0}V_n$. Let $
X\xrightarrow {f_0} X_1\xrightarrow {f_1}X_2 \xrightarrow {f_2}\cdots
$ be an Adams resolution of $X$ as \ref{Adams}. By Proposition \ref{small-big},
    $$
V_n=\{\p\in \Spec(R)\mid (f_{n-1}\circ f_{n-2}\circ \cdots \circ f_0)/1=0\}.
$$
We write $f$ to be $f_{n-1}\circ f_{n-2}\circ \cdots \circ f_0\colon X\rightarrow X_n$. By assumption, $X\in \T^c$. It follows from Lemma \ref{bijective} that $f/1\in \Hom_{\T_\p}(X, X_n)$ is $0$ if and only if $\frac{f}{1}\in \Hom_\T(X,X_n)_\p$ is $0$. Hence,
\begin{align*}
    \Spec(R)\setminus V_n& =\{\p\in \Spec(R)\mid \frac{f}{1}\neq 0\in \Hom_\T(X,X_n)_\p\}\\
    & =\{\p\in \Spec(R)\mid \ann_R(f)\subseteq \p\}
\end{align*}
is closed in $\Spec(R)$, where $\ann_R(f)=\{r\in R\mid r\cdot f=0\}$.
This completes the proof.
\end{proof}
\begin{remark}\label{coincide}
    Let $R$ be a commutative Noetherian ring. Combining with Example \ref{Krause's observation}, it follows from Theorem \ref{main} that, for each $n\geq 0$ and $X,G\in \D^f_b(R)$, the set
    $$
    \{\p\in \Spec(R)\mid \level^{G_\p}_{\D^f_b(R_\p)}(X_\p)\leq n\}
    $$
    is open in $\Spec(R)$. This result is due to Letz \cite[Proposition 3.5]{Letz:2021}. Our Theorem \ref{main} builds on the work of Letz. We observe that the openness result works well for a triangulated subcategory $\C$ within the full subcategory of compact objects in an $R$-linear compactly generated triangulated category $\T$. Letz's proof is based on Oppermann and Šťovíček's work \cite{OS:2012} which establishes the converse coghost lemma for $\D^f_b(R)$, whereas our proof of Theorem \ref{main} makes use of the converse ghost lemma in $\T$; see Lemma \ref{converse-compact}.
\end{remark}

Compare the following result with \cite[Theorem 5.10]{BIK:2015}.
\begin{corollary}\label{local-global principle}
     Keep the assumption as Theorem \ref{main}. For each $X\in \C$ and a full subcategory $\G\subseteq\C$, the following are equivalent.
\begin{enumerate}
    \item $\level^\G_\C(X)<\infty$.

    \item $\level^{\G}_{\C_\p}(X)<\infty$ for each $\p\in \Spec(R)$.

    \item $\level^{\G}_{\C_\m}(X)<\infty$ for each $\m\in \maxSpec(R)$.
\end{enumerate}
\end{corollary}
\begin{proof}
 The implications $(1)\Longrightarrow (2)\Longrightarrow (3)$ are straightforward. Assume $\level^{\G}_{\C_\m}(X)<\infty$ for each $\m\in \maxSpec(R)$. Set $V_n=\{\p\in \Spec(R)\mid \level_{\C_\p}^{\G}(X)\leq n\}$. This is open in $\Spec(R)$ by Theorem \ref{main}.
Since $R$ is Noetherian, the ascending chain
$$
V_1\subseteq V_2\subseteq V_3\subseteq \cdots \subseteq V_n\subseteq \cdots$$ stabilizes for $n\gg 0$. Therefore, there exists $N>0$ such that $V_N=V_i$ for all $i\geq N$. The assumption now yields $\level^{\G_\m}_{\C_\m}(X_\m)\leq N$ for each $\m\in \maxSpec(R)$. By Proposition \ref{small-big}, $\level^\G_\C(X)\leq N$. 
\end{proof}

\begin{chunk}\label{def of singularity category}
    For a commutative Noetherian ring $R$, the \emph{singularity category} of $R$ is the Verdier quotient
$$\D_{\sg}(R)\colonequals \D^f_b(R)/\thick_{\D^f_b(R)}(R).$$
This was first introduced by Buchweitz \cite[Definition 1.2.2]{Buchweitz/Appendix:2021} under the name ``stable derived category"
(see also \cite{Orlov04}).  This terminology is justified by the fact that the singularity category of $R$ is trivial if and only if $R$ is regular.
\end{chunk}
\begin{corollary}\label{application-singularity}
    Let $R$ be a commutative Noetherian ring. For each $n\geq 0$, $X\in \D_{\sg}(R)$, and a full subcategory $\G$ of $\D_{\sg}(R)$, the set
    $$
    \{\p\in \Spec(R)\mid \level^{\G_\p}_{\D_{\sg}(R_\p)}(X_\p)\leq n\}
    $$
    is  open in $\Spec(R)$, where $\G_\p\colonequals\{G_\p\mid G\in \G\}$ is a subcategory of $\D_{\sg}(R_\p)$. 
\end{corollary}

\begin{proof}
   The homotopy category of acyclic complexes of injective $R$-modules, denoted by  $\K_{ac}(\Inj R)$, is $R$-linear. By \cite[Corollary 5.4]{Krause:2005}, $\K_{\ac}(\Inj R)$ is compactly generated, and there is a triangle equivalence 
   $   \D_{\sg}(R)\xrightarrow \simeq \K_{\ac}(\Inj R)^c
    $
    up to direct summands.
    It follows from \cite[Corollary 4.4]{Liu:2023} that the localization functor $\D_{\sg}(R)\rightarrow \D_{\sg}(R_\p)$ induces a triangle equivalence $\D_{\sg}(R)_\p\xrightarrow \simeq \D_{\sg}(R_\p).$  The desired result now follows immediately from Theorem \ref{main}.
\end{proof}

\begin{example}\label{modular-example}
Let $\k$ be a field of characteristic $p>0$, 
and let $G$ be a finite group whose order is divisible by $p$.
We consider the group algebra $\k G$ which is a noncommutative ring,
and its homotopy category of injectives $\T =\K(\Inj \k G)$.
Let $R = \h^*(G,\k)$ be the cohomology ring.
As explained in \cite[Section 10]{BIK:2008} and \cite{BIK:2011annmath,BensonKrause2008},
the ring $R$ is a graded-commutative Noetherian ring,
the category $\T$ is a compactly generated $R$-linear triangulated category,
and there is a triangle equivalence $\D^f_b(\k G)\cong \K(\Inj \k G)^c.$ 
It follows from Theorem \ref{main} that for any $n\geq 0$, any $X\in \D^f_b(\k G)$,
and any full subcategory $\G\subseteq \D^f_b(\k G)$, the set
$$
       \{\p\in \Spec(\h^*(G,\k))\mid \level_{(\D^f_b(\k G))_\p}^{\G}(X)\leq n\}
      $$
is open in $\Spec(\h^*(G,\k))$.
It should be noted that here 
$(\D^f_b(\k G))_\p$ is an abstract localization of the derived category of a noncommutative ring as in \ref{localization at p}, 
so in general it has no simple description. 

Specializing further to the case where $\G = \{\k G\}$,
note that $\level^\G_{\D^f_b(\k G)}(X)<\infty$ if and only if $\pd_{\k G}(X) < \infty$.
It then follows from Corollary \ref{local-global principle}
that $X\in \D^f_b(\k G)$ satisfies $\pd_{\k G}(X) < \infty$
if and only if for all $\p \in \Spec(\h^*(G,\k))$,
it holds that
$\level^{\{\k G\}}_{(\D^f_b(\k G))_\p}(X)<\infty$.
\end{example}

\section{local-global principle on DG rings}\label{local-global section}
 The main result in this section is the local-global principle of $\D^f_b(A)$ over DG rings; see Corollary \ref{local-global for D^f(A)}. This specializes to Letz's local-global principle \cite[Theorem 3.6]{Letz:2021} when $A$ is a commutative Noetherian ring.  Without further assumption, $A$ will be a non-negative commutative DG ring such that $\h_0(A)$ is a Noetherian ring and $\h(A)$ is finitely generated over $\h_0(A)$.
\begin{chunk}\label{localization}
 For each $X,Y\in \D^f_b(A)$ and $\p\in \Spec(A_0)$, by \cite[Lemma 2.7 (iii)]{BSSW} there is a natural isomorphism
$$
\RHom_{A}(X,Y)_\p\xrightarrow \cong \RHom_{A_\p}(X_\p,Y_\p)
$$
in $\D(A_\p)$. 
\end{chunk}

\begin{chunk}\label{semi-free res}
    A DG $A$-module $F$ is called \emph{semi-free} if it admits a filtration of DG $A$-submodules
$$\cdots\subseteq F(i-1) \subseteq  F(i) \subseteq F(i+1) \subseteq\cdots$$
such that $F(i) = 0$ for $i \gg 0$
 , $F = \bigcup_{i\in\Z}F(i)$, and  $F(i)/F(i-1)$ is a direct sum of suspensions of $A$ for each $i\in\Z$; see \cite[Section 8]{AFH} for more details. The filtration $\{F(i)
\}_{i\in\Z}$ is called a \emph{semifree filtration} of $F$.
\end{chunk}

\begin{lemma}\label{dense}
 For each $\p\in \Spec(A_0)$, the canonical functor $\D^f_b(A)\rightarrow \D^f_b(A_\p)$ which maps $X$ to $X_\p$ is dense. 
\end{lemma}

\begin{proof}
For each $Y\in \D^f_b(A_\p)$, by \cite[Proposition B.2]{AINSW} there is a semi-free resolution $G\xrightarrow \simeq Y$ over $A_\p$ such that $G_i=0$ for $i<\inf(Y)$ and $G$ has a semi-free filtration $\{G(i)\}_{i\in \Z}$, where each $G(i)$ fits into a
 short exact sequence 
$
0\rightarrow G(i-1)\rightarrow G(i)\rightarrow \Sigma^i A_\p^{\beta_i}\rightarrow 0
$
of DG $A_\p$-modules 
for some integer $\beta_i\geq 0$.
We claim that there exists a semi-free filtration $\{F(i)\}_{i\in\Z}$ over $A$
such that the following three conditions hold.
\begin{enumerate}
\item  $F(i)=0$ for $i<\inf(Y)$.

\item  Each $F(i)$ fits into a short exact sequence 
$
0\rightarrow F(i-1)\rightarrow F(i)\rightarrow \Sigma^i A^{\beta_i}\rightarrow 0
$
of DG $A$-modules.

\item  For each $i$, there exists an isomorphism $\pi_i\colon F(i)_\p\xrightarrow \cong G(i)$ in $\D^f_b(A_\p)$. Moreover, for each $i$, there is a commutative diagram 
$$
\xymatrix{
F(i-1)_\p\ar[r]\ar[d]_{\pi_{i-1}}^\cong & F(i)_\p\ar[d]^{\pi_i}_\cong\\
G(i-1)\ar[r] & G(i),
}
$$
where the horizontal maps are maps induced by $F(i-1)\subseteq F(i)$ and $G(i-1)\subseteq G(i)$ respectively.  
\end{enumerate}

We choose $F(i)$ which satisfies (1), (2), and (3) by induction on $i$. For $i<\inf(Y)$, set $F(i)=0$. In this case, (2) and (3) hold trivially. Assume we have chosen $F(i)$ for $i<n$ which satisfies (2) and (3). Next, we construct $F(n)$. By induction, there is an isomorphism $\pi_{n-1}\colon F(n-1)_\p\cong G(n-1)$ in $\D^f_b(A_\p)$. By the filtration ${G(i)}$, there is an exact triangle

$$
\Sigma^{n-1}A_\p^{\beta_{n}}\xrightarrow{\varphi_{n-1}} G(n-1)\rightarrow G(n)\rightarrow\Sigma^n A_\p^{\beta_n}
$$
in $\D^f_b(A_\p)$. Let $\K(A)$ denote the homotopy category of DG $A$-modules. 
Consider the isomorphisms
$$
\Hom_{\K(A)}(\Sigma^{n-1}A^{\beta_n}, F_{n-1})_\p\xrightarrow \cong \Hom_{\D(A)}(\Sigma^{n-1}A^{\beta_n}, F_{n-1})_\p\xrightarrow\cong \Hom_{\D(A_\p)}(\Sigma^{n-1}A^{\beta_n}_\p, (F_{n-1})_\p),
$$
where the first isomorphism is because $\Sigma^{n-1}A^{\beta_n}$ is semi-free and the second one is from \ref{localization}. Thus, there exists a morphism $\frac{\alpha_{n-1}}{s_{n-1}}$ in $\Hom_{\K(A)}(\Sigma^{n-1}A^{\beta_n}, F_{n-1})_\p$, where $\alpha_{n-1}\colon \Sigma^{n-1}A^{\beta_n}\rightarrow F_{n-1}$ is in $\K(A)$ and $s_{n-1}\notin \p$, such that the image of $\frac{\alpha_{n-1}}{s_{n-1}}$ under the composition of the above isomorphisms is $(\pi_{n-1})^{-1}\circ \varphi_{n-1}$. We define $F(n)$ to be the cone of $\alpha_{n-1}$. That is,  there is an exact triangle 
$$
\Sigma^{n-1}A^{\beta_n}\xrightarrow{\alpha_{n-1}} F(n-1)\rightarrow F(n)\rightarrow \Sigma^n A^{\beta_n}
$$
in $\K(A)$, and hence there is a short exact sequence of DG $A$-modules
$$
0\rightarrow F(n-1)\rightarrow F(n)\rightarrow \Sigma^n A^{\beta_n}\rightarrow 0.
$$
Note that there exists a morphism $\pi_n\colon F(n)_\p\xrightarrow \cong G(n)$ in $\D^f_b(A_\p)$ such that the following diagram
$$
\xymatrix{
\Sigma^{n-1}A_\p^{\beta_n}\ar[r]^{\alpha_{n-1}/1}\ar[d]^{s_{n-1}/1}_\cong & F(n-1)_\p\ar[r]\ar[d]^{\pi_{n-1}}_\cong &  F(n)_\p \ar[r]\ar@{-->}[d]^{\exists \pi_n}& \Sigma^n A^{\beta_n}_\p\ar[d]^{\Sigma (s_{n-1}/1)}\\
\Sigma^{n-1}A_\p^{\beta_{n}}\ar[r]^{\varphi_{n-1}}& G(n-1)\ar[r]& G(n)\ar[r]& \Sigma A_\p^{\beta_n}
}
$$
in $\D^f_b(A_\p)$ is commutative. Note that $\pi_n$ is an isomorphism. This finishes the inductive proof.

Set $F\colonequals\displaystyle\bigcup_{n\in\mathbb Z} F(i)$. Then $\{F(i)\}_{i\in\Z}$ is a semi-free filtration of $F$. Note that $\{F(i)_\p\}_{i\in\Z}$ is a semi-free filtration of $F_\p$. Combining with (3), we conclude that $F_\p\cong G$ in $\D^f_b(A_\p)$, and hence $F_\p\cong Y$ in $\D^f_b(A_\p)$. 
We define $F^\prime\colonequals\displaystyle \coprod_{i\in \mathbb Z}F^\prime_i$ to be: $F^\prime_i=F_i$ for $i>\sup(Y)$, $F^\prime_i$ is defined to be the image of $\partial_i^F$ for $i=\sup(Y)$, and $F^\prime_i= 0$ for $i<\sup(Y)$; $\del^{F^\prime}= \del^F$. By \cite[Proposition 3.10]{ABIM:2010}, $F^\prime$ is a DG $A$-submodule of $F$. 

Now, consider the DG $A$-module $F/F^\prime$. It suffices to show that $F/F^\prime\in\D^f_b(A)$ and $(F/F^\prime)_\p\cong Y$ in $\D^f_b(A_\p)$.  For each $i\in\Z$, note that $F(i)_j=F_j$ for $j\leq i$ and $F(i)\in \D^f_b(A)$.  It follows that $\h_i(F)$ is finite generated over $\h_0(A)$ for each $i\in\Z$. Since $(F/F^\prime)_i=0$ for $i>\sup(Y)$, $\h_i(F)=\h_i(F/F^\prime)$ for $i\leq \sup(Y)$, and $F_i=0$ for $i<\inf(Y)$, we have $F/F^\prime\in\D^f_b(A)$. 
On the other hand, for each $i>\sup(Y)$, $\h_i(F^\prime_\p)=\h_i(F_\p)=0$, where the second equality is due to the isomorphism $F_\p\cong Y$ in $\D^f_b(A)$. Combining this with $F^\prime_i=0$ for $i<\sup(Y)$ and $\h_i(F^\prime)=0$ for $i=\sup(Y)$,  we conclude that $F^\prime_\p$ is an acyclic DG module, and hence $(F/F^\prime)_\p\cong F_\p\cong Y$ in $\D^f_b(A)$. This completes the proof.
\end{proof}

%\begin{remark}\label{multiplicatively closed-dense}
 %   Let $S$ be a multiplicatively closed subset of $A_0$. With the same argument in the proof of Lemma \ref{dense}, we can prove the canonical functor
%$$
%\D^f_b(A)\rightarrow \D^f_b(S^{-1}A); X\mapsto S^{-1}X
%$$
%is dense.
%\end{remark}

For a DG ring $A$, let  $\K(\Inj A)$ denote the homotopy category of graded-injective DG $A$-modules; a DG $A$-module is \emph{graded-injective} provided that it is injective in the category of graded $A^\natural$-modules.  Note that $\K(\Inj A)$ is an $A_0$-linear triangulated category. A DG $A$-module $I$ is \emph{semi-injective} if $I$ is graded-injective and $\Hom_A(-,I)$ preserves quasi-isomorphisms; see \cite[Section 10]{AFH}. Each DG $A$-module admits a semi-injective resolution; see \cite[Theorem 3.3.3]{AFH}.
\begin{lemma}\label{compact}
   Assume, in addition, $A^\natural$ is Noetherian. For each $M$ in $\D^f_b(A)$, each semi-injective resolution of $M$ is compact in $\K(\Inj A)$,
\end{lemma}
\begin{proof}
    We may assume that $M^\natural$ is finitely generated over $A_0$; indeed, one can take a semi-free resolution which is bounded below and degreewise finite generated over $A_0$ (\cite[Proposition B.2]{AINSW}), then take a good truncation of this semi-free resolution. In particular, $|M|_i=0$ for $i\gg 0$. By \cite[Theorem 12.3.2)]{AFH}, there exists a semi-injective resolution $\epsilon\colon M\xrightarrow \simeq I$ such that $I_i=0$ for $i\gg 0$. Since any two semi-injective resolutions of $M$ are homotopy equivalent, it remains to show that $I$ is compact in $\K(\Inj A)$.

    Let $\Cone(\epsilon)$ denote the cone of $\epsilon$. By the choice of $M$ and $I$, $(\Cone(\epsilon))_i=0$ for $i\gg 0$ and acyclic. By \cite[Theorem 3.4.1 (a)]{Positselski}, $\Cone(\epsilon)$ is coacyclic, and hence $\epsilon$ is isomorphic in the coderived category of DG $A$-modules; refer to \cite[3.3]{Positselski} for the definition of the coderived category. Since $M^\natural$ is finitely generated over $A^\natural$, it follows from \cite[Theorem 3.11.2]{Positselski} that $M$ is compact in the coderived category of DG $A$-modules. Thus, $I$ is compact in $\K(\Inj A)$ as the coderived category of DG $A$-modules is equivalent to $\K(\Inj A)$; see \cite[Theorem 3.7]{Positselski}. 
\end{proof}
\begin{remark}\label{fullyfaithful}
(1) The Noetherian assumption on $A^\natural$ in Lemma \ref{compact} is needed, as it is required in the cited work \cite{Positselski}.

  (2) Let $A$ be a DG ring as Lemma \ref{compact}. The homotopy category of semi-injective DG module is equivalent to
$\D(A)$. Combining with Lemma \ref{compact}, there is a fully faithful functor $i\colon \D^f_b(A)\rightarrow \K(\Inj A)^c$ induced by
taking semi-injective resolution. 
The functor $i$ is an equivalence if $A$ is a left Noetherian ring; see \cite[Proposition 2.3]{Krause:2005}.
\end{remark}
\begin{chunk}\label{relation A_0 H_0(A)}
Set $\h_0(A)=A_0/I$. Note that every prime ideal of $\h_0(A)$ is of the form $\p/I$ for some $\p\in \Spec(A_0)$ containing $I$. Let $\p$ be a prime ideal $A_0$. If $I\subseteq \p$, the ideal $\bar\p\colonequals \p/I$ is a prime ideal of $\h_0(A)$, and $X_{\bar\p}\cong X_\p$ for each $X\in \D^f_b(A)$. If $I\nsubseteq \p$, then $X_\p=0$ for each $X\in \D^f_b(A)$. 
\end{chunk}

Note that the derived category $\D(A)$ is both $A_0$-linear and $\h_0(A)$-linear. 
\begin{proposition}\label{D^f(A)localization}
Assume, in addition, $A^\natural$ is Noetherian. Then: 
\begin{enumerate}
    \item For each $\p\in \Spec(A_0)$, we have
$$
\D^f_b(A)_\p=\D^f_b(A)/\{X\in\D^f_b(A)\mid X_\p=0 \text{ in }\D^f_b(A_\p)\}\cong\D^f_b(A_\p).
$$

\item For each $\bar\p\in \Spec(\h_0(A))$, $\D^f_b(A)_{\bar \p}\cong \D^f_b(A)_\p\cong \D^f_b(A_\p)\cong \D^f_b(A_{\bar \p})$, where $\p\colonequals\pi^{-1}(\bar \p)$ and $\pi\colon A_0\rightarrow \h_0(A)$ is the natural map. 
\end{enumerate}
\end{proposition}

\begin{proof}
(1) For each $X\in \D^f_b(A)$, combining with Remark \ref{fullyfaithful} (2), it follows from Lemma \ref{iso} and \ref{localization} that $X\in \D^f_b(A)_{Z(\p)}$ if and only if $X_\p=0$ in $\D^f_b(A_\p)$. This yields the first equality. Moreover, the localization functor $\D^f_b(A)\rightarrow \D^f_b(A_\p)$ induces a triangle functor $\pi\colon \D^f_b(A)_\p\rightarrow \D^f_b(A_\p)$. By Lemma \ref{iso}, \ref{localization} and Remark \ref{fullyfaithful}, $\pi$ is fully faithful. Lemma \ref{dense} yields that $\pi$ is dense. Hence, $\pi$ is an equivalence.

(2) For each $X\in \D^f_b(A)$, we have $X_{\bar\p}\cong X_\p$. By the definition of localization in triangulated categories, $\D^f_b(A)_{\bar \p}\cong \D^f(A)_\p$. The second equivalence in the statement follows from (1), and the third is a consequence of the isomorphism $A_{\bar \p}\cong A_\p$.
\end{proof}
\begin{remark}
    When $A$ is a commutative Noetherian ring, as noted in Example \ref{Krause's observation}, Proposition \ref{D^f(A)localization} was established by Matsui \cite[Lemma 3.2 (2)]{Matsui:2021}.
\end{remark}
For each $\p\in \Spec(A_0)$ (resp. $\p\in \Spec(\h_0(A))$) and a full subcategory $G\subseteq \D^f_b(A)$, set $\G_\p\colonequals \{G_\p\mid G\in \G\}$ (resp. $\G_{\bar\p}\colonequals \{G_{\bar\p}\mid G\in \G\}$) to be the full subcategory of $\D^f_b(A_\p)$ (resp. $\D^f_b(A_{\bar\p})$).
\begin{corollary}\label{H_0}
      Assume, in addition, $A^\natural$ is Noetherian.
For each $n\geq 0$, $X\in \D^f_b(A)$, and a full subcategory $\G\subseteq \D^f_b(A)$,  the set $\{\bar{\p} \in \Spec(\h_0(A))\mid X_{\bar \p}\in \thick^n_{\D^f_b(A_{\bar\p})}(\G_{\bar \p})\}$ is open in $\Spec(\h_0(A))$. 
 %    \item The set $\{\p \in \Spec(A_0)\mid X_{\p}\in \thick^n_{\D^f_b(A_{\p})}(\G_{\p})\}$ is open in $\Spec(A_0)$. 
\end{corollary}
\begin{proof}
  Note that $\K(\Inj A)$ is an $A_0$-linear triangulated category. Moreover, $\K(\Inj A)$ is compactly generated; see \cite[Theorem 3.7 and 3.11.2]{Positselski}. By Remark \ref{fullyfaithful},  $\D^f_b(A)$ can be regarded as a triangulated subcategory of $\K(\Inj A)^c$. Combining with Proposition \ref{D^f(A)localization} and Theorem \ref{main}, the set $\{\p \in \Spec(A_0)\mid X_{\p}\in \thick^n_{\D^f_b(A_{\p})}(\G_{\p})\}$ is open in $\Spec(A_0)$. 
On the other hand, the canonical map $\pi\colon A_0\rightarrow \h_0(A)$ induces a continuous map
$$
\pi^\ast\colon \Spec(\h_0(A))\rightarrow \Spec(A_0);~ \bar\p\mapsto \pi^{-1}(\bar \p).
$$
For each $\bar\p\in \Spec(\h_0(A))$, set $\p=\pi^{-1}(\bar \p)$. By Proposition \ref{D^f(A)localization}, 
$$\{\bar\p \in \Spec(\h_0(A))\mid X_{\bar\p}\in \thick^n_{\D^f_b(A_{\bar\p})}(\G_{\bar\p})\}=(\pi^\ast)^{-1}\{\p \in \Spec(A_0)\mid X_{\p}\in \thick^n_{\D^f_b(A_{\p})}(\G_{\p})\}.$$
Thus, $\{\bar\p \in \Spec(\h_0(A))\mid X_{\bar\p}\in \thick^n_{\D^f_b(A_{\bar\p})}(\G_{\bar\p})\}$ is open.
\end{proof}

\begin{corollary}\label{local-global for D^f(A)}
    Assume, in addition, $A^\natural$ is Noetherian. For each $X\in \D^f_b(A)$ and a full subcategory $\G\subseteq \D^f_b(A)$, the following are equivalent.
\begin{enumerate}
    \item $\level^\G_{\D^f_b(A)}(X)<\infty$.

 %   \item $\level^{\G_{\p}}_{\D^f_b(A_\p)}(X_\p)<\infty$ for each $\p\in \Spec(A_0)$.
    
  \item $\level^{\G_{\bar\p}}_{\D^f_b(A_{\bar\p})}(X_{\bar\p})<\infty$ for each $\bar\p\in \Spec(\h_0(A))$.
  
 %   \item $\level^{\G_\m}_{\D^f_b(A_\m)}(X_\m)<\infty$ for each $\m\in \maxSpec(A_0)$.

    \item $\level^{\G_{\bar\m}}_{\D^f_b(A_{\bar\m})}(X_{\bar\m})<\infty$ for each $\bar\m\in \maxSpec(\h_0(A))$.
\end{enumerate}
\end{corollary}
\begin{proof}
    As noted in the proof of Corollary \ref{H_0}, the $A_0$-linear triangulated category $\K(\Inj A)$ is compactly generated, and $\D^f_b(A)$ can be regarded as a triangulated subcategory of $\K(\Inj A)^c$. Thus, by Corollary \ref{local-global principle}, the condition $(1)$ holds if and only if $\level^{\G_{\p}}_{\D^f_b(A_\p)}(X_\p)<\infty$ for each $\p\in \Spec(A_0)$, which is further equivalent to requiring $\level^{\G_\m}_{\D^f_b(A_\m)}(X_\m)<\infty$ for each $\m\in \maxSpec(A_0)$.

    $(1)\iff (2). $ Set $\h_0(A)=A_0/I$ and let $\p$ be a prime ideal of $A_0$. If $I\nsubseteq \p$, then $M_\p$ is acyclic for each $M\in\D^f_b(A)$. In this case, $\level^{\G_{\p}}_{\D^f_b(A_\p)}(X_\p)<\infty$ always holds. 
    If $I\subseteq \p$, then $\bar\p=\p/I$ is a prime ideal of $\h_0(A)$. By Proposition \ref{D^f(A)localization}, $\level^{\G_{\p}}_{\D^f_b(A_\p)}(X_\p)<\infty$ if and only if $\level^{\G_{\bar\p}}_{\D^f_b(A_{\bar\p})}(X_{\bar\p})<\infty$. Note that each prime ideal of $\h_0(A)$ is of the form $\p/I$ for some prime ideal $\p$ of $A_0$ containing $I$. Therefore, combining the above with the equivalences from the first paragraph of the proof,  we conclude that $(1)\iff (2)$. The same argument shows $(1)\iff (3)$.
\end{proof}

We finish this section with an example showing that even for ordinary rings, such results are false without some finiteness assumption, explaining why our main results assume objects are compact.

\begin{example}
Let $A$ be a commutative Noetherian ring of infinite Krull dimension.
According to \cite[Proposition 5.4]{Bass1962},
for every $n \in \mathbb{N}$,
there exists an $A$-module $M_n$ such that $\pd_A(M_n) = n$.
Let $M = \bigoplus_{n=1}^{\infty} M_n$.
The isomorphism $\Hom_A(M,-) \cong \prod_n \Hom_A(M_n,-)$ shows that $\pd_A(M) = +\infty$.
Let $\p \in \Spec(A)$.
For any $n \in \mathbb{N}$,
since $\pd_A(M_n) = n <\infty$,
it follows that $\pd_{A_{\p}}((M_n)_{\p}) < \infty$.
By Krull's Hauptidealsatz, the Noetherian local ring $A_{\p}$ has finite Krull dimension, so  
\cite[Theorem II.3.2.6]{RG1971} implies that
$\pd_{A_{\p}}((M_n)_{\p}) \le \dim(A_{\p})$.
Since this holds for any $n \in \mathbb{N}$,
we deduce that
$\pd_{A_{\p}}(M_{\p}) \le \dim(A_{\p}) < \infty$.
Let $\G$ be the collection of all free $A$-modules.
We have thus shown that 
$\level^{\G_\p}_{\D_b(A_{\p})}(M_{\p})<\infty$ for each $\p\in \Spec(A)$,
but $\level^\G_{\D_b(A)}(M) = +\infty$.
\end{example}

\section{A DG version of a result of Gabber}\label{Section: DG of Gabber}

 The purpose of this section is a DG version of a result of Gabber (\ref{injective locally finite}); see Theorem \ref{DG case of Gabber}. This result will be applied in Section \ref{openness section}. The proof adapts the arguments of Avramov, Iyengar, and Lipman \cite[Proposition 1.5]{Avramov-Iyengar-Lipman} to the DG setting. Throughout this section, $A$ is a non-negative commutative DG ring such that $\h_0(A)$ is Noetherian and $\h(A)$ is finitely generated over $\h_0(A)$.

\begin{chunk}\label{injective locally finite}
    Let $R$ be a commutative Noetherian ring. For each complex $X$ of $R$-modules, let $\id_R(X)$ denote its injective dimension of $X$ over $R$, as defined by Avramov and Foxby \cite[Definition 2.1.I]{Avramov-Foxby:1991}. Given $M\in\D^f_b(R)$, the following two conditions are equivalent, as shown by Gabber \cite[Proposition 1.5]{Avramov-Iyengar-Lipman}.
\begin{enumerate}
    \item For each $N\in \D^f_b(R)$, $\RHom_R(N,M)\in \D^f_b(R)$.

    \item $\id_{R_\p}(M_\p)<\infty$ for each $\p\in \Spec(R)$.
\end{enumerate}

For each $\p\in \Spec(R)$, there exists a maximal ideal $\m$ of $R$ containing $\p$. Since $X_\p$ is the localization of $X_\m$ at $\p R_\m$ for each $R$-module $X$, it follows from \cite[Lemma 3.2.5]{BH} that the above two conditions are equivalent to requiring that
\begin{itemize}
    \item $\id_{R_\m}(M_\m)<\infty$ for each $\m\in \maxSpec(R)$.
\end{itemize}

Moreover, if $\dim(R)<\infty$, then 
$\id_R(M)<\infty$ if and only if $\id_{R_\p}(M_\p)<\infty$ for each prime ideal $\p$ of $R$. The forward direction is by \cite[Corollary 3.1.3]{BH}. For the backward direction, since $\id_{R_\p}(M_\p)$ is finite for each $\p\in \Spec(R)$ and $\dim(R)<\infty$, $\sup\{\id_{R_\p}(M_\p)\mid\p\in \Spec(R)\}$ is finite (see \cite[Theorem 3.1.17]{BH}), and hence $\id_R(M)<\infty$ (see \cite[Proposition 5.3.I]{Avramov-Foxby:1991}).
    \end{chunk}

%The following result is a DG version of \ref{injective locally finite}, and its proof builds on the work of Avramov, Iyengar, and Lipman \cite[Proposition 1.5]{Avramov-Iyengar-Lipman}.

The formula in this section may look slightly different from that in \cite{Shaul2018, Shaul-TAMS} as we work with homological grading.
\begin{chunk}\label{injective-transfer}
In \cite[Definition A]{Shaul2018}, the third author extended the definition of the injective dimension of a complex introduced by Avramov and Foxby \cite[Definition 2.1.I]{Avramov-Foxby:1991} to DG settings.  
Specifically, for a DG $A$-module $X$, the \emph{injective dimension} of $X$ over $A$, denoted by  $\id_A(X)$, is defined to be  
   $$
\inf\{n\in \mathbb Z\mid \Ext^i_A(Y,X)=0 \text{ for each } Y\in \D_b(A) \text{ and each }i>n+\sup(Y)\}
;$$ see also \cite{BSSW}.
Moreover, the third author \cite[Theorem 2.5]{Shaul2018} observed that the injective dimension over $A$ could be computed via the injective dimension over $\h_0(A)$:
$$
\id_A(X)=\id_{\h_0(A)}(\RHom_A(\h_0(A), X)).
$$

\end{chunk}

\begin{theorem}\label{DG case of Gabber}
Let $A$ be a non-negative commutative DG ring such that $\h_0(A)$ is Noetherian and  $\h(A)$ is finitely generated over $\h_0(A)$. For each $X\in \D^f_b(A)$, consider the following conditions:
    \begin{enumerate}
        \item $\id_A(X)<\infty$.

        \item For each $Y\in \D^f_b(A)$, $\RHom_A(Y, X)\in \D^f_b(A)$.

       % \item $\id_{A_{\p}}(X_\p)<\infty$ for each $\p\in \Spec(A_0)$.

        \item $\id_{A_{\bar\p}}(X_{\bar\p})<\infty$ for each $\bar\p\in \Spec(\h_0(A))$.

      %  \item $\id_{A_{\m}}(X_\m)<\infty$ for each $\m\in \maxSpec(A_0)$.

        \item $\id_{A_{\bar\m}}(X_{\bar\m})<\infty$ for each $\bar\m\in \maxSpec(\h_0(A))$.
    \end{enumerate}
Then we have $(1)\Longrightarrow (2)\iff(3)\iff (4)$. If, in addition, $\dim(\h_0(A))<\infty$, then the implication $(2)\Longrightarrow (1)$ also holds. 
\end{theorem}

The proof of  Theorem \ref{DG case of Gabber} will be given at the end of this section. Before starting the proof, some preliminaries are needed. A key part of the proof is Proposition \ref{lem:Rfd-is-finite}.

\begin{chunk}  
Assume, in addition, $A$ is local, meaning that $\h_0(A)$ is a local ring; see \ref{local DG ring}.
For each $M \in \D^f_b(A)$, following \cite[Definition 3.1]{Shaul-TAMS}
the depth of  $M$ is defined to be
$$
\depth_A(M) = -\sup(\RHom_A(\k,M)),
$$
where $\k$ is the residue field of $\h_0(A)$.
This is related to the length of a maximal $M$-regular sequence,
called the \emph{sequential depth} of $M$ in \cite{Shaul-TAMS}, which is denoted by \( \seqdepth_A(M) \) and given by the formula
$$
\seqdepth_A(M) = \depth_A(M) +\sup(M);
$$
see \cite[Section 5]{Shaul-TAMS} for details.
\end{chunk}

The usual inequalities about the relation between depth and short exact sequences remain true in the DG setting. For example:

\begin{lemma}\label{lem:depthDis}
Assume, in addition, $A$ is local.
 Let $M_1 \to M_2 \to M_3 \to \Sigma M_1$ be an exact triangle in $\D(A)$.
Then $\depth_A(M_2) \ge \min\{\depth_A(M_1), \depth_A(M_3)\}$.
\end{lemma}
\begin{proof}
Set $\k$ as the residue field of $\h_0(A)$. Applying the functor $\RHom_A(\k,-)$ we obtain an exact triangle
\[
\RHom_A(\k,M_1) \to \RHom_A(\k,M_2) \to \RHom_A(\k,M_3) \to \Sigma\RHom_A(\k,M_1),
\]
and hence, for each $n \in \mathbb{Z}$, there is an exact sequence of $\h_0(A)$-modules
\[
\Ext^n_A(\k,M_1) \to \Ext^n_A(\k,M_2) \to \Ext^n_A(\k,M_3),
\]
which proves the desired result.
\end{proof}

For each $M \in \D^f_b(A)$, following \cite[(1.01)]{Avramov-Iyengar-Lipman},
 the \emph{restricted flat dimension} of $M$ is defined to be
\[
\Rfd_A(M) := \sup\{\depth_{A_{\bar{\p}}}(A_{\bar{\p}})-\depth_{A_{\bar{\p}}}(M_{\bar{\p}})\mid \bar{\p} \in \Spec(\h_0(A))\}.
\]
The following result is a direct consequence of Lemma \ref{lem:depthDis}.
\begin{lemma}\label{lem:RfdDis}
Let $M_1 \to M_2 \to M_3 \to \Sigma M_1$ be an exact triangle in $\D(A)$.
Then $$\Rfd_A(M_2) \leq\max\{\Rfd_A(M_1), \Rfd_A(M_3)\}$$
\end{lemma}
The following result is a DG version of \cite[Theorem 1.1]{Avramov-Iyengar-Lipman}.

\begin{proposition}\label{lem:Rfd-is-finite}
For each $M \in \D^f_b(A)$.
Then $\Rfd_A(M) <\infty$.
\end{proposition}
\begin{proof}
By taking truncations, any $M\in\D^f_b(A)$ with $\amp(M) > 0$ fits into an exact triangle
\[
M^\prime \to M \to M^{\prime\prime} \to \Sigma M^\prime
\]
in $\D^f_b(A)$ with $\amp(M^\prime) < \amp(M)$ and $\amp(M^{\prime\prime}) < \amp(M)$.
Combining this with Lemma \ref{lem:RfdDis} and induction,
we may assume that $\amp(M) = 0$.
Then, by shifting, we may assume that $M$ is a finitely generated $\h_0(A)$-module.
Assuming by contradiction that there exists $M \in \D^f_b(A)$ such that $\Rfd_A(M) = \infty$.
Arguing as in the proof of \cite[Theorem 1.1]{Avramov-Iyengar-Lipman}, we may assume that $M = \h_0(A)/\bar{\p}$,
where $\bar{\p} \in \Spec(\h_0(A))$,
and that for any ideal $\bar{I}\in\h_0(A)$ strictly containing $\bar{\p}$,
it holds that $\Rfd_A(\h_0(A)/\bar{I}) <\infty$.

Let $\bar{\x}$ be a finite sequence of elements in $\h_0(A)$ that generate the ideal $\bar{\p}$,
and consider the Koszul complex $ K(A;\bar{\x})$ studied in \cite[Section 2]{Shaul2021-Adv}.
Since $K(A;\bar{\x})$ is a non-negative commutative Noetherian DG ring with bounded homology
and $\h_0(K(A;\bar{\x})) = \h_0(A)/{\bar{\p}}$ is an integral domain, 
we can choose an element $\bar{f} \in \h_0(A)$ with $\bar{f} \notin \bar{\p}$
such that the localization $(\h_i(K(A;\bar{\x})))_{\bar{f}}$ is a finitely generated free $\h_0(K(A;\bar{\x}))_{\bar{f}}$-module for all $i$; this follows from the same argument as in the proof of \cite[Proposition 7.22]{Jacobson}.
Let $j = \Rfd_A(\h_0(A)/(\bar{\p},\bar{f})) <\infty$.

We will show that $\Rfd_A(\h_0(A)/\bar{\p}) < \infty$ by following the approach in the proof of \cite[Theorem 1.1]{Avramov-Iyengar-Lipman} and establishing a uniform bound for the  following invariant
\begin{equation}\label{eqn:to-be-bounded}
\depth_{A_{\bar{\q}}}(A_{\bar{\q}})-\depth_{A_{\bar{\q}}}((\h_0(A)/\bar{\p})_{\bar{\q}})
\end{equation}
for all $\bar{\q} \in \Spec(\h_0(A))$. Once the uniform bound is established, it follows that \(\Rfd_A(\h_0(A)/\bar{\p})\) is finite.
We proceed by considering three cases.

Case 1: $\bar{\p} \nsubseteq \bar{\q}$. In this case, we have that 
$(\h_0(A)/\bar{\p})_{\bar{\q}} = 0$,
and hence $\depth_{A_{\bar{\q}}}((\h_0(A)/\bar{\p})_{\bar{\q}}) = \infty$. It follows from this that
the equation (\ref{eqn:to-be-bounded}) is equal to $-\infty$.

Case 2: $\bar{\p} \subseteq \bar{\q}$, and $\bar{f} \in \bar{\q}$.
In this case, since $\h_0(A/\bar{\p})$ is an integral domain,
there is an exact triangle
\[
\h_0(A/\bar{\p}) \xrightarrow{\bar{f}} \h_0(A/\bar{\p}) \to \h_0(A)/(\bar{\p},\bar{f}) \to \Sigma\h_0(A/\bar{\p})
\]
in $\D(A)$,
which implies using Lemma \ref{lem:depthDis}
that the equation (\ref{eqn:to-be-bounded}) is less or equal to $j$.

Case 3: $\bar{\p} \subseteq \bar{\q}$ and $\bar{f} \notin \bar{\q}$.
In this case, as in the proof of \cite[Theorem 1.1]{Avramov-Iyengar-Lipman},
we claim that
\begin{equation}\label{essential equality}
    \depth_{A_{\bar{\q}}}(K(A_{\bar{\q}};\bar{\x})) = \depth_{A_{\bar{\q}}}(\h_0(A)_{\bar{\q}}/\bar{\p}_{\bar{\q}}) -\sup(K(A_{\bar{\q}};\bar{\x})).
\end{equation}
Rearranging this, and using the fact that $\h_0(A)_{\bar{\q}}/\bar{\p}_{\bar{\q}}$ is concentrated in degree $0$,
it is equivalent to the claim that
\[
\seqdepth_{A_{\bar{\q}}}(K(A_{\bar{\q}};\bar{\x})) = \seqdepth_{A_{\bar{\q}}}(\h_0(A)_{\bar{\q}}/\bar{\p}_{\bar{\q}}),
\]
which holds because all homologies of $K(A_{\bar{\q}};\bar{\x})$ are finitely generated free as $\h_0(A)_{\bar{\q}}/\bar{\p}_{\bar{\q}}$-modules.
On the other hand, 
let $\k$ denote the residue field of the local ring
$\h_0(A)_{\bar{\q}}$.
Consider
\begin{align*}
    \depth_{A_{\bar{\q}}}(K(A_{\bar{\q}};\bar{\x})) &= 
-\sup\left( \RHom_{A_{\bar{\q}}}(\k, K(A_{\bar{\q}};\bar{\x})) \right) \\
&=-\sup\left( \RHom_{A_{\bar{\q}}}(\k,A_{\bar{\q}}) \otimes^{\rm L}_{A_{\bar{\q}}} K(A_{\bar{\q}};\bar{\x}) \right) \\
&=-\sup\left( \RHom_{A_{\bar{\q}}}(\k,A_{\bar{\q}}) \otimes^{\rm L}_{\k} \k \otimes^{\rm L}_{A_{\bar{\q}}} K(A_{\bar{\q}};\bar{\x}) \right)  \\
&=\depth_{A_{\bar{\q}}}(A_{\bar{\q}}) -\sup\left(\k \otimes^{\rm L}_{A_{\bar{\q}}} K(A_{\bar{\q}};\bar{\x})\right),
\end{align*}
where the first equality is by definition,
the second one is because the Koszul complex is compact,
the third one is by associativity of the derived tensor product,
and the fourth one is because $\k$ is a field.
The above equalities yield the second equality below:
\begin{align*}
  \depth_{A_{\bar{\q}}}(A_{\bar{\q}})-\depth_{A_{\bar{\q}}}((\h_0(A)/\bar{\p})_{\bar{\q}})& = 
\depth_{A_{\bar{\q}}}(A_{\bar{\q}}) - \depth_{A_{\bar{\q}}}(K(A_{\bar{\q}};\bar{\x}))-\sup(K(A_{\bar{\q}};\bar{\x}))  \\
&=\sup\left(\k \otimes^{\rm L}_{A_{\bar{\q}}} K(A_{\bar{\q}};\bar{\x})\right)  -\sup(K(A_{\bar{\q}};\bar{\x}))\\
&\le
\sup\left(\k \otimes^{\rm L}_{A_{\bar{\q}}} K(A_{\bar{\q}};\bar{\x})\right) \\
&\le \amp(A) + g, 
\end{align*}
where the first equality is by (\ref{essential equality}) and $g$ is the length of the sequence $\bar{\x}$.
This establishes the uniform bound.
\end{proof}

The following result is a DG version of \cite[Lemma 1.4]{Avramov-Iyengar-Lipman}.

\begin{proposition}\label{prop:local-bound-rhom}
Assume, in addition, $A$ is local.
Let $M, N \in \D^f_b(A)$ with
 $\id_A(N) < \infty$.
Then 
\[
-\inf(\RHom_A(M,N)) =  \depth_A(A) - \depth_A(M)-\inf(N).
\]
\end{proposition}
\begin{proof}
By assumption, $\RHom_A(M,N)\in \D^f_b(A)$. Combining this with the proof of \cite[Proposition 3.1]{Yekutieli}, we get the first equality below:
\begin{align*}
    -\inf(\RHom_A(M,N))& = -\inf(\RHom_A(M,N) \otimes^{\rm L}_A \h_0(A)) \\
&=-\inf(\RHom_A(M,N) \otimes^{\rm L}_A \h_0(A) \otimes^{\rm L}_{\h_0(A)} \k) \\ 
&=-\inf(\RHom_A(M,N) \otimes^{\rm L}_A \k),
\end{align*}
where the second equality is by \cite[Lemma A.4.3]{Avramov-Iyengar-Lipman}.
Since $\id_A(N)<\infty$, it follows from  \cite[Lemma 2.7(2)(i)]{BSSW} that
\[
\RHom_A(M,N) \otimes^{\rm L}_A \k \cong
\RHom_A(\RHom_A(\k,M),N),
\]
where $\k$ is the residue field of $\h_0(A)$.
Using the adjunction along the map $A \to \k$, 
we have
\[
\RHom_A(\RHom_A(\k,M),N) \cong
\RHom_{\k}(\RHom_A(\k,M),\RHom_A(\k,N)).
\]
Since $\k$ is a field, 
there is an equality
\begin{align*}
   -\inf\left(\RHom_{\k}(\RHom_A(\k,M),\RHom_A(\k,N))\right)& = 
-\inf(\RHom_A(\k,N))) +\sup(\RHom_A(\k,M)) \\
&=-\inf(\RHom_A(\k,N))) - \depth_A(M). 
\end{align*}
Thus, for any $M \in \D^f_b(A)$, we conclude that
\begin{equation}\label{inf-RHom(M,N)}
    -\inf(\RHom_A(M,N)) = 
-\inf(\RHom_A(\k,N))) - \depth_A(M).
\end{equation}
Applying this to $M = A$,
we get
$
-\inf(N) = 
-\inf(\RHom_A(\k,N))) - \depth_A(A).
$
The desired result follows by combining this with (\ref{inf-RHom(M,N)}).
\end{proof}

\begin{proof}[Proof of Theorem \ref{DG case of Gabber}]
    The implication $(1)\Longrightarrow(2)$ is by definition. 
    The implication $(3)\Longrightarrow (4)$ is trivial. Using the same argument in \ref{injective locally finite}, the implication $(4)\Longrightarrow (3)$ follows from the equality in \ref{injective-transfer}.
    
    $(2)\Longrightarrow (3)$. By assumption, $\RHom_A(\h_0(A),X)$ is in $\D^f_b(A)$, and hence it is in $\D^f_b(\h_0(A))$. For each $M\in \D^f_b(\h_0(A))$, there is an isomorphism
    $$
\RHom_{\h_0(A)}(M,\RHom_A(\h_0(A),X))\cong \RHom_A(M,X)\in\D^f_b(A).
$$
This yields that $\RHom_{\h_0(A)}(M,\RHom_A(\h_0(A),X))\in \D^f_b(\h_0(A))$ for each $M\in \D^f_b(\h_0(A))$. For each $\bar\p\in \Spec(\h_0(A))$, it follows from \ref{injective locally finite} that $\id_{\h_0(A)_{\bar\p}}(\RHom_{A}(\h_0(A),X)_{\bar\p})<\infty$, and hence $\id_{\h_0(A_{\bar\p})}(\RHom_{A_{\bar\p}}(\h_0(A_{\bar\p}),X_{\bar\p}))<\infty$ (see \ref{localization}). By \ref{injective-transfer}, $\id_{A_{\bar\p}}(X_{\bar\p})<\infty$ for each $\bar\p\in \Spec(\h_0(A))$. 

$(3)\Longrightarrow (2)$. It is enough to show that $\RHom_A(Y,X)$ has bounded homology for each $Y\in\D^f_b(A)$. This is equivalent to show that $-\inf(\RHom_A(Y,X))<\infty$.  
To see this, note that for $\bar{\p} \in \Spec(\h_0(A))$,
Proposition \ref{prop:local-bound-rhom} implies that
\begin{align*}
-\inf\left(\RHom_A(Y,X)_{\bar{\p}}\right)&=  
-\inf\left(\RHom_{A_{\bar{\p}}}(Y_{\bar{\p}},X_{\bar{\p}})\right) \\
 &=\depth_{A_{\bar{\p}}}(A_{\bar{\p}}) - \depth_{A_{\bar{\p}}}(Y_{\bar{\p}}) -\inf(X_{\bar{\p}}) \\
 &\le
  \left(\depth_{A_{\bar{\p}}}(A_{\bar{\p}}) - \depth_{A_{\bar{\p}}}(Y_{\bar{\p}}) \right)-\inf(X)
\end{align*}
By Proposition \ref{lem:Rfd-is-finite},
the latter is uniformly bounded by the finite number $\Rfd_A(Y)-\inf(X)$,
which proves the result.

Assume $(2)$ holds and $\dim(\h_0(A))<\infty$. It follows from the argument in the proof of $(2)\Longrightarrow (4)$ that $\id_{\h_0(A)_{\bar\p}}(\RHom_{A}(\h_0(A),X)_{\bar\p})<\infty$ for each $\bar\p\in \Spec(\h_0(A))$. Note that the assumption yields $\RHom_A(\h_0(A),X)\in \D^f_b(\h_0(A))$.  Combining this with $\dim(\h_0(A))<\infty$, it follows from \ref{injective locally finite} that $\id_{\h_0(A)}(\RHom_A(\h_0(A),X))<\infty$, and hence (1) holds by \ref{injective-transfer}. 
\end{proof}

\section{Openness of the finite injective dimension locus}\label{openness section}
The main result in this section is Theorem \ref{T1} from the introduction which concerns the openness of the finite injective dimension locus over DG rings; see Theorem \ref{FID open}. Without further assumption,  $A$ will be assumed to be a non-negative commutative DG ring such that $\h_0(A)$ is Noetherian and $\h(A)$ is finitely generated over $\h_0(A)$.

\begin{chunk}\label{def of point-wise dualizing}
 We say that a DG $A$-module $D\in \D^f_b(A)$ is \emph{point-wise dualizing} if it satisfies the following conditions:
   \begin{enumerate}
       \item For each $X\in \D^f_b(A)$, $\RHom_A(X,D)\in \D^f_b(A)$.

       \item For each $X\in \D^f_b(A)$, the biduality morphism $X\rightarrow\RHom_A(\RHom_A(X,D),D)$ is an isomorphism in $\D(A)$. 
   \end{enumerate}
A point-wise dualizing DG $A$-module $D$ is said to be \emph{dualizing} provided that $\id_A(D)<\infty$.

 % For example, for a Gorenstein ring $R$ (not necessarily finite Krull dimension), the ring $R$ itself is point-wise dualizing (see \cite[Proposition 1.5 and Theorem 2.1]{Avramov-Iyengar-Lipman}); recall that a commutative Noetherian ring $R$ is called \emph{Gorenstein} if the injective dimension of $R_\p$ as $R_\p$-module is finite for each prime ideal $\p$ of $R$. If $\dim(\h_0(A))<\infty$, then a point-wise dualizing DG $A$-module has finite injective dimension by Theorem \ref{DG case of Gabber}.
\end{chunk}
\begin{remark}\label{compare definitions}
     (1) Let $M\in \D^f_b(A)$. According to Yekutieli \cite[Definition 7.1]{Yekutieli},  $M$ is  dualizing if $\id_A(D)<\infty$ and there is a natural isomorphism $A\rightarrow \RHom_A(D,D)$ in $\D(A)$. In our DG setting of $A$, $M$ is dualizing in the sense of \ref{def of point-wise dualizing} if and only if it is dualizing in the sense of Yekutieli; see \cite[Proposition 7.17]{Yekutieli}.

   Moreover, in our DG setting of $A$, if $M$ is dualizing in the sense of \ref{def of point-wise dualizing}, then it is also dualizing in the sense of Frankild, Iyengar, and J{\o}rgensen \cite[Definition 1.8]{FIJ} (see \cite[Proposition 7.17]{Yekutieli}). Conversely, a dualizing DG module in the sense of \cite{FIJ}  is point-wise dualizing according to our definition. 
   
  (2) If, in addition, $\dim(\h_0(A))<\infty$, then Theorem \ref{DG case of Gabber} shows that a point-wise dualizing DG module is dualizing in the sense of \ref{def of point-wise dualizing}, and hence all definitions of “dualizing DG modules” mentioned in \cite{FIJ}, \cite{Yekutieli}, and \ref{def of point-wise dualizing} coincide in this case.

(3) Let $R$ be a commutative Noetherian ring and $M\in\D^f_b(R)$. Then $M$ is a (resp. point-wise) dualizing complex in the sense of \ref{def of point-wise dualizing} if and only if it is a (resp. point-wise) dualizing complex in the sense of Grothendieck and Hartshorne \cite[Chapter V]{Hartshorne}; the equivalence of the point-wise case follows from \ref{injective locally finite}.  
%Moreover, a complex $M\in \D^f_b(R)$ is point-wise dualizing if and only if $M$ is point-wise dualizing in the sense of \cite[Chapter V]{Hartshorne} (namely, $M_\p$ is dualizing over $R_\p)$; see Proposition \ref{point-wise point-wise dualizing}.
\end{remark}
\begin{example}\label{compare example}
    Let $R$ be a Gorenstein ring $R$ with infinite Krull dimension. As noted in \ref{def of point-wise dualizing}, $R$ is point-wise dualizing. However, $\id_R(R)=\infty$; see \cite[Theorem 3.1.17]{BH}. Therefore, $R$ is not dualizing in the sense of \cite{Hartshorne} and \cite{Yekutieli}. Such rings exist even in the context of regular rings, as shown by Nagata \cite[Appendix, Example 1]{Nagata}.
\end{example}
\begin{lemma}\label{locally biduality}
    Let  $X,Y$ be DG modules in $\D^f_b(A)$. If $\RHom_A(X,Y)\in \D^f_b(A)$, then the following conditions are equivalent.
    \begin{enumerate}
        \item The biduality morphism $X\rightarrow \RHom_A(\RHom_A(X,Y),Y)$ is an isomorphism in $\D(A)$.
        
%\item The biduality morphism $X_{\p}\rightarrow \RHom_{A_{\p}}(\RHom_{A_{\p}}(X_{\p},Y_{\p}),Y_{\p})$ is an isomorphism in $\D(A_\p)$ for each $\p\in \Spec(A_0)$.

        \item The biduality morphism $X_{\bar\p}\rightarrow \RHom_{A_{\bar\p}}(\RHom_{A_{\bar\p}}(X_{\bar\p},Y_{\bar\p}),Y_{\bar\p})$ is an isomorphism in $\D(A_{\bar\p})$ for each $\bar\p\in \Spec(\h_0(A))$.

   %     \item The biduality morphism $X_{\m}\rightarrow \RHom_{A_{\m}}(\RHom_{A_{\m}}(X_{\m},Y_{\m}),Y_{\m})$ is an isomorphism in $\D(A_\m)$ for each $\m\in \maxSpec(A_0)$.

        \item The biduality morphism $X_{\bar\m}\rightarrow \RHom_{A_{\bar\m}}(\RHom_{A_{\bar\m}}(X_{\bar\m},Y_{\bar\m}),Y_{\bar\m})$ is an isomorphism in $\D(A_{\bar\m})$ for each $\bar\m\in \maxSpec(\h_0(A))$.
    \end{enumerate} 
\end{lemma}
\begin{proof}
Since $\RHom_A(X,Y)\in \D^f_b(A)$, the implication $(1)\Longrightarrow (2)$ is by \ref{localization} and \ref{relation A_0 H_0(A)}. The implication $(2)\Longrightarrow (3)$ is trivial. It remains to prove $(3)\Longrightarrow (1)$.
For each $\bar\m\in \maxSpec(\h_0(A))$, 
    \begin{align*}
      \RHom_{A}(\RHom_A(X,Y),Y)_{\bar\m}
          &\cong \RHom_{A_{\bar\m}}(\RHom_{A}(X,Y)_{\bar\m},Y_{\bar\m})\\
          & \cong\RHom_{A_{\bar\m}}(\RHom_{A_{\bar\m}}(X_{\bar\m},Y_{\bar\m}),Y_{\bar\m}),
    \end{align*}
      where the first isomorphism is by \ref{localization} and $\RHom_A(X,Y)\in \D^f_b(A)$, and the second one is by \ref{localization}. The desired result of $(1)$ follows by combining the fact: For each complex $M$ of $\h_0(A)$-modules, $M$ is acyclic if and only if $M_{\bar\m}$ is acyclic for each $\bar\m\in \maxSpec(\h_0(A))$.
\end{proof}
\begin{proposition}\label{point-wise point-wise dualizing}
      For each $X\in\D^f_b(A)$, the following conditions are equivalent.
      \begin{enumerate}
          \item $X$ is point-wise dualizing over $A$.

          %\item  $X_{\p}$ is dualizing over $A_{\p}$ for each $\p\in \Spec(A_0)$.
          
          \item  $X_{\bar\p}$ is dualizing over $A_{\bar\p}$ for each $\bar\p\in \Spec(\h_0(A))$.

         %  \item  $X_{\m}$ is dualizing over $A_{\m}$ for each $\m\in \maxSpec(A_0)$.
          
          \item  $X_{\bar\m}$ is dualizing over $A_{\bar\m}$ for each $\bar\m\in \maxSpec(\h_0(A))$.
      \end{enumerate}
\end{proposition}
\begin{proof}
$(1)\Longrightarrow (2)$. Assume $X$ is point-wise dualizing over $A$. For each $\bar\p\in \Spec(\h_0(A))$ and $Y\in \D^f_b(A_{\bar\p})$, there exists $Y^\prime\in \D^f_b(A)$ such that $Y^\prime_{\bar\p}\cong Y$ in $\D^f_b(A_{\bar\p})$; see Lemma \ref{dense} and \ref{relation A_0 H_0(A)}.  According to the definition of the point-wise dualizing DG module, 
$\RHom_A(Y^\prime,X)\in \D^f_b(A)$, and hence $$\RHom_{A_{\bar\p}}(Y,X_{\bar\p})\cong \RHom_A(Y^\prime,X)_{\bar\p}\in \D^f_b(A_{\bar\p});$$ the isomorphism is from \ref{localization}. By Lemma \ref{locally biduality},
   the biduality morphism $ Y^\prime_{\bar\p}\rightarrow \RHom_{A_{\bar\p}}(\RHom_{A_{\bar\p}}(
Y^\prime_{\bar\p},X_{\bar\p}),X_{\bar\p})$ is an isomorphism.
Thus, $X_{\bar\p}$ is point-wise dualizing over $A_{\bar\p}$, and hence it is dualizing over $A_{\bar\p}$ by Remark \ref{compare definitions} (2). 

The implication $(2)\Longrightarrow (3)$ is trivial. 
  
$(3)\Longrightarrow (1)$. Let $Y$ be a DG module in $\D^f_b(A)$.  By assumption, $\id_{A_{\bar\m}}(X_{\bar\m})<\infty$ for each $\bar\m\in \maxSpec(\h_0(A))$. Combining this with Theorem \ref{DG case of Gabber},  we get $\RHom_A(Y,X)\in \D^f_b(A)$. Consequently, Lemma \ref{locally biduality} implies that the biduality morphism $Y \to \RHom_A(\RHom_A(Y, X), X)$ is an isomorphism in $\D(A)$. Thus, $X$ is point-wise dualizing over $A$.
\end{proof}
\begin{chunk}\label{ess finite type-dualizing}
   Let $R$ be a commutative Noetherian ring.
Assume $R$ has a dualizing complex. Then any $R$-algebra essentially of finite type over $R$ has a dualizing complex; see for example \cite[0A7K, Proposition 15.11]{StacksProject}.
\end{chunk}
\begin{chunk}\label{ess finite type-point-wise dualizing}
     Let $R$ be a commutative Noetherian ring with a point-wise dualizing complex. Then any $R$-algebra essentially of finite type over $R$ has a point-wise dualizing complex.
  
  This statement can be proved by following the same argument as the proof of \ref{ess finite type-dualizing} (see \cite[0A7K, Lemma 15.6, 15.9 and 15.10]{StacksProject}). Lemma 15.6 and 15.9 of \cite[0A7K]{StacksProject} can be adapted to our context similarly. It remains to establish Lemma 15.10 of \cite[0A7K]{StacksProject} in our context: assume $R$ has a point-wise dualizing complex $M\in \D^f_b(R)$, then $M[x]\colonequals M\otimes_R R[x]$ is a point-wise dualizing complex of $R[x]$. We prove this in the following:
  
  First, we show that $\id_{R[x]_\q}(M[x]_\q)<\infty$ for each $\q\in \Spec(R[x])$. 
  Let $\q$ be a prime ideal in $\Spec(R[x])$, and set $\p\colonequals\q\cap R$ (this is a prime ideal of $R$). By assumption, $\id_{R_\p}(M_\p)<\infty$. Since $M[x]_\p\cong M_\p\otimes_{R_\p}R_\p[x]$, it follows from \cite[0A6J Lemma 3.10]{StacksProject} that $\id_{R[x]_\p}(M[x]_\p)<\infty$. Let $U$ be the image of $R[x]\setminus \q$ under the map $R[x]\rightarrow R[x]_\p$. Note that $U$ is a multiplicatively closed subset of $R[x]_\p$, and $U^{-1}(M[x]_\p)\cong M[x]_\q$.
  Therefore, by \cite[Corollary 3.1.3]{BH}, $\id_{R[x]_\q}(M[x]_\q)<\infty$; see also \cite[0A6I Lemma 3.8]{StacksProject}. 
  The isomorphism $R[x]\cong \RHom_{R[x]}(M\otimes_R R[x],M\otimes_R R[x])$ is obtained as follows: 
  \begin{align*}
       R[x]\cong \RHom_R(M,M)\otimes_R R[x] 
       \cong \RHom_R(M,M[x])
       \cong \RHom_{R[x]}(M[x],M[x]),
  \end{align*}
where the second isomorphism is by \cite[Theorem 8.4.12]{CFH}, and the third one is by the adjunction. Thus, $M[x]$ is point-wise dualizing over $R[x]$.
\end{chunk}
\begin{chunk}\label{projective-transfer}
   Let $X$ be a DG $A$-module. Bird, the third author, Sridhar, and Williamson \cite[Definition 2.1]{BSSW} defined the \emph{projective dimension} of $X$ over $A$, denoted by  $\pd_A(X)$, to be  
   $$
\inf\{n\in \mathbb Z\mid \Ext^i_A(X,Y)=0 \text{ for each } Y\in \D_b(A) \text{ and each }i>n-\inf(Y)\}
.$$
This extended the definition of the projective dimension of a complex over a ring introduced by Avramov and Foxby \cite[Definition 2.1.P]{Avramov-Foxby:1991}. Moreover, similarly as the injective dimension (see \ref{injective-transfer}), it is proved in \cite[Corollary 2.3]{BSSW} that
$$
\pd_A(X)=\pd_{\h_0(A)}(\h_0(A)\otimes_A^{\rm L}X).
$$
\end{chunk}

The next lemma is contained in \cite[Theorem 14.1.33]{Yekutieli:book}. We give an alternate proof using our methods developed above.

\begin{lemma}\label{characterizatiof perfect}
   Assume, in addition, $A^\natural$ is Noetherian.  For each $X\in \D^f_b(A)$, $X\in \thick_{\D^f_b(A)}(A)$ if and only if $\pd_A(X)<\infty$. 
\end{lemma}
\begin{proof}
        By \ref{projective-transfer}, we get the first and last equivalences below:
        \begin{align*}
            \pd_A(X)<\infty& \iff \pd_{\h_0(A)}(\h_0(A)\otimes^{\rm L}_A X)<\infty\\
            &\iff \h_0(A)\otimes^{\rm L}_A X\in \thick_{\D^f_b(\h_0(A))}(\h_0(A))\\
            & \iff \h_0(A_{\bar\p})\otimes^{\rm L}_{A_{\bar\p}}X_{\bar\p}\in \thick_{\D^f_b(\h_0(A_{\bar\p}))}(\h_0(A_{\bar\p})), \forall ~\bar{\p}\in \Spec(\h_0(A))\\
            & \iff \pd_{\h_0(A_{\bar\p})}(\h_0(A_{\bar\p})\otimes^{\rm L}_{A_{\bar\p}}X_{\bar\p}) <\infty, \forall~ \bar{\p}\in \Spec(\h_0(A))\\
& \iff \pd_{A_{\bar\p}}(X_{\bar\p})<\infty, \forall~ \bar{\p}\in \Spec(\h_0(A)),
        \end{align*}
        where the second and the fourth equivalences follow from\cite[1.1]{Buchweitz/Appendix:2021},  and the third one is by
       Corollary \ref{local-global for D^f(A)}. 
        On the other hand, Corollary \ref{local-global for D^f(A)} implies that 
        $$X\in \thick_{\D^f_b(A)}(A)\iff X_{\bar\p}\in \thick_{\D^f_b(A_{\bar\p})}(A_{\bar\p}), \forall~\bar\p\in \Spec(\h_0(A)).$$ Thus, we may assume $\h_0(A)$ is local. 

Assume $\h_0(A)$ is local. If $X\in \thick_{\D(A)}(A)$, then $\h_0(A)\otimes^{\rm L}_A X\in \thick_{\D^f_b(\h_0(A))}(\h_0(A))$; see \ref{level}. It follows that $\pd_{\h_0(A)}(\h_0(A)\otimes^{\rm L}_AX)<\infty$, and hence $\pd_A(X)<\infty$ by \ref{projective-transfer}. For the converse, assume $\pd_A(X)<\infty$, it follows from the definition that $\Ext^i_A(X,\k)=0$ for $i\gg 0$, where $\k$ is the residue field of $\h_0(A)$. By \cite[Remark B.10]{AINSW}, $X\in \thick_{\D^f_b(A)}(A)$. 
\end{proof}

The proof of the above result contains the following, which is a differential graded version of a result of Bass and Murthy (\cite[Lemma 4.5]{BM67}:
\begin{corollary}
 Assume, in addition, $A^\natural$ is Noetherian.   
 For each $X\in \D^f_b(A)$, it holds that $\pd_A(X)<\infty$ if and only if 
 $\pd_{A_{\bar\p}}(X_{\bar\p})<\infty$ for each $\bar\p\in \Spec(\h_0(A))$
 if and only if $\pd_{A_{\bar\m}}(X_{\bar\m})<\infty$ for each $\bar\m\in \maxSpec(\h_0(A)$.
\end{corollary}

As the definition of the perfect complex in \cite[1.1]{Buchweitz/Appendix:2021}, we say a DG $A$-module $X\in \D^f_b(A)$ is \emph{perfect} if $X\in \thick_{\D^f_b(A)}(A)$. 
See \cite[Chapter 14]{Yekutieli:book} and \cite[Theorem 4.16]{Shaul-Williamson} for related results about perfect DG modules.

\begin{lemma}\label{Hom of injective}
 Assume, in addition, $A^\natural$ is Noetherian. For each $X,Y\in \D^f_b(A)$, if $\id_{A_{\bar\m}}(X_{\bar\m})<\infty$ and $\id_{A_{\bar\m}}(Y_{\bar\m})<\infty$ for each $\bar\m\in \maxSpec(\h_0(A))$, then $\RHom_A(X,Y)$ is a perfect DG $A$-module.
\end{lemma}
\begin{proof}
 By Theorem \ref{DG case of Gabber}, $\RHom_A(X,Y)\in \D^f_b(A)$, and hence $\RHom_{A_{\bar\m}}(X_{\bar\m},Y_{\bar\m})\in \D^f_b(A_{\bar\m})$ for each $\bar\m\in \maxSpec(\h_0(A))$.   Combining this with Corollary \ref{local-global for D^f(A)}, we may assume $\id_A(X)<\infty$ and $\id_A(Y)<\infty$. 
 
Next, assume $\id_A(X)<\infty$ and $\id_A(Y)<\infty$. We claim that the flat dimension (see the definition in \cite[Definition 2.1]{BSSW}) of $\h_0(A)\otimes^{\rm L}_A \RHom_A(X,Y)$ over $\h_0(A)$ is finite. Consider
\begin{align*}
 \h_0(A)\otimes^{\rm L}_A \RHom_A(X.Y)& \cong \RHom_A(\RHom_A(\h_0(A), X),Y)\\
 & \cong \RHom_{\h_0(A)}(\RHom_A(\h_0(A),X),\RHom_A(\h_0(A),Y)),
\end{align*}
where the first isomorphism is by \cite[Lemma 2.7]{BSSW}, and the second one is by adjunction.  Note that both $\RHom_A(\h_0(A),X)$ and $\RHom_A(\h_0(A),Y)$ have finite injective dimension over $\h_0(A)$. Recall that $\Hom_{\h_0(A)}(I,J)$ is flat over $\h_0(A)$ for injective $\h_0(A)$-modules $I,J$; see \cite[Proposition 1.4.9]{CFH}. It follows that $\RHom_{\h_0(A)}(\RHom_A(\h_0(A),X),\RHom_A(\h_0(A),Y))$ has finite flat dimension over $\h_0(A)$, and hence $\h_0(A)\otimes^{\rm L}_A \RHom_A(X, Y)$ has finite flat dimension over $\h_0(A)$. 
Note that $\h_0(A)\otimes^{\rm L}_A \RHom_A(X,Y)\in \D^f(\h_0(A))$ as $\RHom_A(X,Y)\in \D^f_b(A)$. 
We conlcude that 
$$\pd_{\h_0(A)}(\h_0(A)\otimes^{\rm L}_A \RHom_A(X,Y))<\infty.$$ 
Combining with \ref{projective-transfer}, we have $\pd_A(\RHom_A(X,Y))<\infty$, and hence $\RHom_A(X,Y)$ is perfect over $A$ by Lemma \ref{characterizatiof perfect}.
\end{proof}

\begin{proposition}\label{characterization of finite injective dim}
   Assume, in addition, $A^\natural$ is Noetherian and that $A$ has a point-wise dualizing DG module $D$. For each $X\in \D^f_b(A)$, the following conditions are equivalent.
    \begin{enumerate}
        \item $X\in \thick_{\D^f_b(A)}(D)$.
        
%\item $\id_{A_{\p}}(X_{\p})<\infty$ for each $\p\in \Spec(A_0)$.

        \item   $\id_{A_{\bar\p}}(X_{\bar\p})<\infty$ for each $\bar\p\in \Spec(\h_0(A))$.

 %  \item $\id_{A_{\p}}(X_{\m})<\infty$ for each $\m\in \maxSpec(A_0)$.

   \item $\id_{A_{\bar\m}}(X_{\bar\m})<\infty$ for each $\bar\m\in \maxSpec(\h_0(A))$.
    \end{enumerate}
    If, in addition, $\dim(\h_0(A))<\infty$, then $X\in \thick_{\D^f_b(A)}(D)$ if and only if $\id_A(X)<\infty$. 
\end{proposition}
\begin{proof}
   $(1)\Longrightarrow (2)$. Assume $X\in \thick_{\D^f_b(A)}(D)$. This implies that $X_{\bar\p}\in \thick_{\D^f_b(A_{\bar\p})}(D_{\bar\p})$ for each $\bar\p\in \Spec(\h_0(A))$. Let $\bar\p$ be a prime ideal of $\h_0(A)$. By Proposition \ref{point-wise point-wise dualizing}, $\id_{A_{\bar\p}}(D_{\bar\p})<\infty$.  Note that the full subcategory of $\D^f_b(A_{\bar\p})$ consisting of DG modules with finite injective dimensions over $A_{\bar\p}$ is thick. Since $X_{\bar\p}\in \thick_{\D^f_b(A_{\bar\p})}(D_{\bar\p})$, we get $\id_{A_{\bar\p}}(X_{\bar\p})<\infty$.

The implication $(2)\Longrightarrow (3)$ is trivial. 
   
   $(3)\Longrightarrow (1)$. Assume $\id_{A_{\bar\m}}(X_{\bar\m})<\infty$ for each $\bar\m\in \Spec(\h_0(A))$. Since $D$ is point-wise dualizing, we get $\RHom_{A}(X,D)\in \D^f_b(A)$. In particular, $\RHom_{A_{\bar\m}}(X_{\bar\m},D_{\bar\m})\in \D^f_b(A_{\bar\m})$ for each $\bar\m\in \Spec(\h_0(A))$; see \ref{localization}. Thus, Lemma \ref{Hom of injective} yields that $\RHom_{A_{\bar\m}}(X_{\bar\m},D_{\bar\m})$ is a perfect DG $A_{\bar\m}$-module for each $\bar\m\in \Spec(\h_0(A))$.  It follows from Corollary \ref{local-global for D^f(A)} that $\RHom_A(X,D)$ is perfect over $A$, and hence $\RHom_A(\RHom_A(X,D),D)$ is in $\thick_{\D^f_b(A)}(D)$; see \ref{level}. This implies that $X$ is in $\thick_{\D^f_b(A)}(D)$ as $X\cong \RHom_A(\RHom_A(X,D),D)$.

Next, we prove the second statement. Assume $\id_A(X)<\infty$, then it follows from the first statement and Theorem \ref{DG case of Gabber} that $X\in \thick_{\D^f_b(A)}(D)$. Assume $X\in \thick_{\D^f_b(A)}(D)$ and $\dim(\h_0(A))<\infty$. By Remark \ref{compare definitions} (2), $D$ is dualizing, and hence $\id_A(D)<\infty$. Combining this with $X\in \thick_{\D^f_b(A)}(D)$, we conclude that $\id_A(X)<\infty$. 
\end{proof}
\begin{chunk}\label{FID}
    Let $X$ be a DG $A$-module. The \emph{finite injective dimension locus} ($\FID$-locus) of $X$ 
    over $\h_0(A)$ is defined to be the set $$\FID_{\h_0(A)}(X)\colonequals\{\bar\p\in \Spec(\h_0(A))\mid \id_{A_{\bar\p}}(X_{\bar \p})<\infty\}.$$
\end{chunk}

\begin{theorem}\label{FID open}
 Let $A$ be a non-negative commutative DG ring such that $A^\natural$ is Noetherian and $\h(A)$ is finitely generated over $\h_0(A)$.   Assume $A$ has a point-wise dualizing DG module. For each $X\in \D^f_b(A)$,  and $\FID_{\h_0(A)}(X)$ is open in $\Spec(\h_0(A))$.
\end{theorem}
\begin{proof}
  By assumption, there exists a point-wise dualizing DG $A$-module $D$. Let $\bar\p$ be a prime ideal of $\h_0(A)$. By Proposition \ref{point-wise point-wise dualizing}, $D_{\bar\p}$ is dualizing over $A_{\bar\p}$. It follows from Proposition \ref{characterization of finite injective dim} that $\id_{A_{\bar\p}}(X_{\bar\p})<\infty$ if and only if $X_{\bar\p}\in \thick_{\D^f_b(A_{\bar\p})}(D_{\bar\p})$, and hence we get the first equality below:
    \begin{align*}
    \FID_{\h_0(A)}(X)& =\{\bar\p\in \Spec(\h_0(A))\mid X_{\bar\p}\in \thick_{\D^f_b(A_{\bar\p})}(D_{\bar\p})\}  \\
    & =\bigcup_{n\geq 0}\{\bar\p\in \Spec(\h_0(A)\mid X_{\bar\p}\in \thick^n_{\D^f_b(A_{\bar\p}))}(D_{\bar\p})\}.
    \end{align*}
This set is open in $\Spec(\h_0(A))$ by Corollary \ref{H_0}. 
%With the same argument,  we can prove $\FID_{\h_0(A)}(X)$ is open. This can be also proved by using that $\FID_{A_0}(X)$ is open. Indeed, $\FID_{\h_0(A)}(X)=(\pi^\ast)^{-1}(\FID_{A_0}(X))$, where $\pi^\ast\colon \Spec(\h_0(A))\rightarrow \Spec(A_0)$ is the continuous map induced by the canonical map $\pi\colon A_0\rightarrow \h_0(A)$. Combining with that $\FID_{A_0}(X)$ is open, we get that $\FID_{\h_0(A)}(X)$ is open in $\Spec(\h_0(A))$.
\end{proof}
\begin{remark}
(1) One can also define the finite injective dimension locus of a DG $A$-module $X$ over $A_0$ as $$\FID_{A_0}(X)\colonequals\{\p\in \Spec(A_0)\mid \id_{A_{\p}}(X_{\p})<\infty\}.$$
Keep the same assumption as Theorem \ref{FID open}, for each $X\in \D^f_b(A)$, the same argument of Theorem \ref{FID open} will imply that $\FID_{A_0}(X)$ is open in $\Spec(A_0)$. 

(2) Takahashi \cite[Theorem]{Takahashi:2006_Glasgow} observed that the $\FID$-locus of a finitely generated module over an excellent ring is open. 
\end{remark}
\begin{corollary}\label{essentially finite type-open}
   Let $R$ be a commutative Noetherian ring with a point-wise dualizing complex. Assume $S$ is an $R$-algebra essentially finite type over $R$. For each $M\in \D^f_b(S)$, $\FID_S(M)$ is open in $\Spec(S)$. 
\end{corollary}
\begin{proof}
   By \ref{ess finite type-point-wise dualizing}, $S$ has a point-wise dualizing complex.  The desired result follows from Theorem \ref{FID open}.
\end{proof}

\begin{remark}\label{ess finite type}
Let $R$ be a Gorenstein ring and $S$ be an $R$-algebra essentially of finite type over $R$. For each $M\in \D^f_b(S)$, Corollary \ref{essentially finite type-open} implies that $\FID_S(M)$ is open.  When $M$ is a finitely generated $S$-module, this result is due to Kimura \cite[Proposition 3.9 and Corollary 4.2]{Kimura}. 
\end{remark}

\begin{chunk}\label{Gorenstein def}
     If, in addition, $A$ is local,  $A$ is said to be \emph{Gorenstein} if $A$ itself serves as a dualizing DG module (equivalently, $\id_A(A)<\infty$). In general, $A$ is said to be \emph{Gorenstein} provided that the local DG ring $A_{\bar \p}$ is Gorenstein for each $\bar\p\in \Spec(\h_0(A))$. 

Following \cite{Shaul2022}, the \emph{Gorenstein locus} of $A$ over over $\h_0(A)$ is defined to be $$\Gor_{\h_0(A)}(A)\colonequals\{\bar\p\in \Spec(\h_0(A))\mid A_{\bar \p} \text{ is a Gorenstein DG ring}\}.$$
\end{chunk}
 \begin{remark}
 By Proposition \ref{point-wise point-wise dualizing},   $A$ is Gorenstein if and only if it is point-wise dualizing as a DG $A$-module. In our DG ring setting, the definition of a Gorenstein DG ring coincides with that in \cite{FIJ, FJ}.
 \end{remark}

Note that $\Gor_{\h_0(A)}(A)=\FID_{\h_0(A)}(A)$. As an immediate consequence of Theorem \ref{FID open}, we have:
\begin{corollary}\label{Gorenstein loci}
 Let $A$ be a non-negative commutative DG ring such that $A^\natural$ is Noetherian and $\h(A)$ is finitely generated over $\h_0(A)$.   Assume $A$ has a point-wise dualizing DG module. Then $\Gor_{\h_0(A)}(A)$ is open in $\Spec(\h_0(A))$.
\end{corollary}

\begin{remark}\label{recover Shaul's result}
When $A$ has a dualizing DG module in the sense of \cite[Definition 7.1]{Yekutieli}, Corollary \ref{Gorenstein loci} is due to the third author \cite[Theorem 1]{Shaul2022}. Specifically, the third author demonstrated that, for a non-negative DG ring such that $\h_0(A)$ is Noetherian and $\h(A)$ is finitely generated over $\h_0(A)$, the set $\Gor_{\h_0(A)}(A)$ is open when $A$ has a dualizing DG module in the sense of \cite{Yekutieli}. 

\end{remark}

Given a DG module $X$, its support is given by
\[
\Supp_{A}(X) = \bigcup_{n\in \mathbb{Z}} \Supp_{\h_0(A)}(\h_n(X)).
\]
In particular, if $X \in \D^f_b(A)$ and $\h_0(A)$ is Noetherian then $\Supp_A(X)$ is a closed subset of $\Spec(\h_0(A))$.

\begin{remark}
As shown in \cite[Proposition 5]{Shaul2022},
even in very nice cases, the open set $\Gor_{\h_0(A)}(A)$ may be empty. 
If $A^\natural$ is Noetherian, $\h(A)$ is finitely generated over $\h_0(A)$, and $A$ has a point-wise dualizing DG module,
and assuming $X$ is perfect over $A$,
and that $\Gor_{\h_0(A)}(A) \cap \Supp_{A}(X) \ne \emptyset$, it follows that the open set $\FID_{\h_0(A)}(X)$ is non-empty.
\end{remark}

The notion of a Cohen-Macaulay DG ring was introduced in \cite[Definition 8.8]{Shaul-TAMS}. We have the following observation about the Cohen-Macaulay locus.   

\begin{proposition}
Let $A$ be a non-negative commutative DG ring such that $A^\natural$ is Noetherian and $\h(A)$ is finitely generated over $\h_0(A)$. Assume $A$ has a point-wise dualizing DG module,
and let $M \in \D^f_b(A)$ be such that $\amp(M_{\bar\p}) \le \amp(A_{\bar\p})$ for all $\bar\p\in\Supp(\h_0(A))$ and $\Supp_A(M) = \Spec(\h_0(A))$.
Then the set
\[
\CM(A) = \{\bar\p \in \Spec(\h_0(A)) \mid A_{\bar\p} \mbox{ is Cohen-Macaulay}\}
\]
contains the open set $\FID_{\h_0(A)}(M)$.
\end{proposition}
\begin{proof}
This follows from \cite[Theorem 5.22(2)]{Shaul-TAMS}.
\end{proof}

\bibliographystyle{amsplain}
\bibliography{mainbib}

\providecommand{\bysame}{\leavevmode\hbox to3em{\hrulefill}\thinspace}
\providecommand{\MR}{\relax\ifhmode\unskip\space\fi MR }
\providecommand{\MRhref}[2]{%
  \href{http://www.ams.org/mathscinet-getitem?mr=#1}{#2}
}
\providecommand{\href}[2]{#2}
\begin{thebibliography}{10}

\bibitem{ABIM:2010}
Luchezar~L. Avramov, Ragnar-Olaf Buchweitz, Srikanth~B. Iyengar, and Claudia Miller, \emph{Homology of perfect complexes}, Adv. Math. \textbf{223} (2010), no.~5, 1731--1781.

\bibitem{Avramov-Foxby:1991}
Luchezar~L. Avramov and Hans-Bj{\o}rn Foxby, \emph{Homological dimensions of unbounded complexes}, J. Pure Appl. Algebra \textbf{71} (1991), no.~2-3, 129--155.

\bibitem{AFH}
Luchezar~L. Avramov, Hans-Bjørn Foxby, and Stephen Halperin, \emph{Differential graded homological algebra}, 2003.

\bibitem{Avramov-Iyengar-Lipman}
Luchezar~L. Avramov, Srikanth~B. Iyengar, and Joseph Lipman, \emph{Reflexivity and rigidity for complexes, {I}: {Commutative} rings}, Algebra Number Theory \textbf{4} (2010), no.~1, 47--86.

\bibitem{AINSW}
Luchezar~L. Avramov, Srikanth~B. Iyengar, Saeed Nasseh, and Sean Sather-Wagstaff, \emph{Homology over trivial extensions of commutative {DG} algebras}, Commun. Algebra \textbf{47} (2019), no.~6, 2341--2356.

\bibitem{Bass1962}
Hyman Bass, \emph{Injective dimension in {N}oetherian rings}, Trans. Amer. Math. Soc. \textbf{102} (1962), 18--29. \MR{138644}

\bibitem{BM67}
Hyman Bass and M.~Pavaman Murthy, \emph{Grothendieck groups and {P}icard groups of abelian group rings}, Ann. of Math. (2) \textbf{86} (1967), 16--73.

\bibitem{Beligiannis:2008}
Apostolos Beligiannis, \emph{Some ghost lemmas}, Survey for The Representation Dimension of Artin Algebras, Bielefeld. Preprint (2008).

\bibitem{BIK:2008}
Dave Benson, Srikanth~B. Iyengar, and Henning Krause, \emph{Local cohomology and support for triangulated categories}, Ann. Sci. {\'E}c. Norm. Sup{\'e}r. (4) \textbf{41} (2008), no.~4, 575--621.

\bibitem{BIK:2011annmath}
\bysame, \emph{Stratifying modular representations of finite groups}, Ann. Math. (2) \textbf{174} (2011), no.~3, 1643--1684.

\bibitem{BIK:2011}
\bysame, \emph{Stratifying triangulated categories}, J. Topol. \textbf{4} (2011), no.~3, 641--666.

\bibitem{BIK:2015}
\bysame, \emph{A local-global principle for small triangulated categories}, Math. Proc. Camb. Philos. Soc. \textbf{158} (2015), no.~3, 451--476.

\bibitem{BensonKrause2008}
David~John Benson and Henning Krause, \emph{Complexes of injective {$kG$}-modules}, Algebra Number Theory \textbf{2} (2008), no.~1, 1--30. \MR{2377361}

\bibitem{BSSW}
Isaac Bird, Liran Shaul, Prashanth Sridhar, and Jordan Williamson, \emph{Finitistic dimensions over commutative {DG}-rings}, Math. Z. \textbf{309} (2025), no.~1, 29.

\bibitem{Bondal/vdB:2003}
Alexei Bondal and Michel van~den Bergh, \emph{Generators and representability of functors in commutative and noncommutative geometry}, Mosc. Math. J. \textbf{3} (2003), no.~1, 1--36.

\bibitem{Brodmann-Sharp}
Markus~P. Brodmann and Rodney~Y. Sharp, \emph{Local cohomology. {An} algebraic introduction with geometric applications}, 2nd ed. ed., Camb. Stud. Adv. Math., vol. 136, Cambridge: Cambridge University Press, 2012.

\bibitem{BH}
Winfried Bruns and J{\"u}rgen Herzog, \emph{Cohen-{Macaulay} rings}, rev. ed. ed., Camb. Stud. Adv. Math., vol.~39, Cambridge: Cambridge University Press, 1998.

\bibitem{Buchweitz/Appendix:2021}
Ragnar-Olaf Buchweitz, \emph{Maximal {Cohen}-{Macaulay} modules and {Tate} cohomology. {With} appendices by {Luchezar} {L}. {Avramov}, {Benjamin} {Briggs}, {Srikanth} {B}. {Iyengar} and {Janina} {C}. {Letz}}, Math. Surv. Monogr., vol. 262, Providence, RI: American Mathematical Society (AMS), 2021.

\bibitem{Christensen:1998}
J.~Daniel Christensen, \emph{Ideals in triangulated categories: {Phantoms}, ghosts and skeleta}, Adv. Math. \textbf{136} (1998), no.~2, 284--339.

\bibitem{CFH}
Lars~Winther Christensen, Hans-Bj{\o}rn Foxby, and Henrik Holm, \emph{Derived category methods in commutative algebra}, Springer Monogr. Math., Springer, 2024.

\bibitem{FIJ}
Anders Frankild, Srikanth~B. Iyengar, and Peter J{\o}rgensen, \emph{Dualizing differential graded modules and {Gorenstein} differential graded algebras.}, J. Lond. Math. Soc., II. Ser. \textbf{68} (2003), no.~2, 288--306.

\bibitem{FJ}
Anders Frankild and Peter J{\o}rgensen, \emph{Gorenstein differential graded algebras}, Isr. J. Math. \textbf{135} (2003), 327--353.

\bibitem{GM1978}
Silvio Greco and Maria~Grazia Marinari, \emph{Nagata's criterion and openness of loci for {Gorenstein} and complete intersection}, Math. Z. \textbf{160} (1978), 207--216.

\bibitem{Grothendieck:1965}
Alexander Grothendieck, \emph{\'el\'ements de g\'eom\'etrie alg\'ebrique : {IV.} {\'etude} locale des sch\'emas et des morphismes de sch\'emas}, vol.~24, Institut des Hautes \'Etudes Scientifiques, 1965.

\bibitem{Hartshorne}
Robin Hartshorne, \emph{Residues and duality. {Lecture} notes of a seminar on the work of {A}. {Grothendieck}, given at {Havard} 1963/64. {W}ith an appendix by {P}. {Deligne}}, Lect. Notes Math., vol.~20, Springer, 1966.

\bibitem{Iyengar/Takahashi:2016}
Srikanth~B. Iyengar and Ryo Takahashi, \emph{Annihilation of cohomology and strong generation of module categories}, Int. Math. Res. Not. \textbf{2016} (2016), no.~2, 499--535.

\bibitem{Iyengar/Takahashi:2019}
\bysame, \emph{Openness of the regular locus and generators for module categories}, Acta Math. Vietnam. \textbf{44} (2019), no.~1, 207--212.

\bibitem{Jacobson}
Nathan Jacobson, \emph{Basic algebra {II}.}, 2nd ed. ed., New York, NY: W. H. Freeman {and} Company, 1989.

\bibitem{Kimura}
Kaito Kimura, \emph{Openness of various loci over {Noetherian} rings}, J. Algebra \textbf{633} (2023), 403--424.

\bibitem{Krause:2005}
Henning Krause, \emph{The stable derived category of a noetherian scheme}, Compos. Math. \textbf{141} (2005), no.~5, 1128--1162.

\bibitem{Letz:2020}
Janina~C. Letz, \emph{Generation time in derived categories}, Ph.D. thesis, The University of Utah, 2020.

\bibitem{Letz:2021}
\bysame, \emph{Local to global principles for generation time over commutative {Noetherian} rings}, Homology Homotopy Appl. \textbf{23} (2021), no.~2, 165--182.

\bibitem{Leuschke:2002}
Graham~J. Leuschke, \emph{Gorenstein modules, finite index and finite {Cohen}-{Macaulay} type.}, Commun. Algebra \textbf{30} (2002), no.~4, 2023--2035.

\bibitem{Liu:2023}
Jian Liu, \emph{Annihilators and dimensions of the singularity category}, Nagoya Math. J. \textbf{250} (2023), 533--548.

\bibitem{Matsui:2021}
Hiroki Matsui, \emph{Prime thick subcategories and spectra of derived and singularity categories of {Noetherian} schemes}, Pac. J. Math. \textbf{313} (2021), no.~2, 433--457.

\bibitem{Nagata1959}
Masayoshi Nagata, \emph{On the closedness of singular loci}, Publ. Math., Inst. Hautes {\'E}tud. Sci. \textbf{2} (1959), 1--12.

\bibitem{Nagata}
\bysame, \emph{Local rings}, Intersci. Tracts Pure Appl. Math., vol.~13, Interscience Publishers, New York, NY, 1962.

\bibitem{Neeman-norm}
Amnon Neeman, \emph{The connection between the {K}-theory localization theorem of {Thomason}, {Trobaugh} and {Yao} and the smashing subcategories of {Bousfield} and {Ravenel}}, Ann. Sci. {\'E}c. Norm. Sup{\'e}r. (4) \textbf{25} (1992), no.~5, 547--566.

\bibitem{Neeman-book}
\bysame, \emph{Triangulated categories}, Ann. Math. Stud., vol. 148, Princeton, NJ: Princeton University Press, 2001.

\bibitem{OS:2012}
Steffen Oppermann and Jan Šťovíček, \emph{Generating the bounded derived category and perfect ghosts}, Bull. Lond. Math. Soc. \textbf{44} (2012), no.~2, 285--298.

\bibitem{Orlov04}
Dmitri Orlov, \emph{Triangulated categories of singularities and {D}-branes in {Landau}-{Ginzburg} models}, Proc. Steklov Inst. Math. \textbf{246} (2004), 227--248.

\bibitem{Positselski}
Leonid Positselski, \emph{Two kinds of derived categories, {Koszul} duality, and comodule-contramodule correspondence}, Mem. Am. Math. Soc., vol. 996, Providence, RI: American Mathematical Society (AMS), 2011.

\bibitem{RG1971}
Michel Raynaud and Laurent Gruson, \emph{Crit\`eres de platitude et de projectivit\'e. {T}echniques de ``platification'' d'un module}, Invent. Math. \textbf{13} (1971), 1--89. \MR{308104}

\bibitem{sharp}
Rodney~Y. Sharp, \emph{Acceptable rings and homomorphic images of {Gorenstein} rings}, J. Algebra \textbf{44} (1977), 246--261.

\bibitem{Shaul2018}
Liran Shaul, \emph{Injective {DG}-modules over non-positive {DG}-rings}, J. Algebra \textbf{515} (2018), 102--156.

\bibitem{Shaul-TAMS}
\bysame, \emph{The {Cohen}-{Macaulay} property in derived commutative algebra}, Trans. Am. Math. Soc. \textbf{373} (2020), no.~9, 6095--6138.

\bibitem{Shaul2021-Adv}
\bysame, \emph{Koszul complexes over {C}ohen-{M}acaulay rings}, Adv. Math. \textbf{386} (2021), Paper No. 107806, 35.

\bibitem{Shaul2022}
\bysame, \emph{Open loci results for commutative {DG}-rings}, J. Pure Appl. Algebra \textbf{226} (2022), no.~5, 11, Id/No 106922.

\bibitem{Shaul-Williamson}
Liran Shaul and Jordan Williamson, \emph{Lifting (co)stratifications between tensor triangulated categories}, Israel J. Math. \textbf{261} (2024), no.~1, 249--280. \MR{4776493}

\bibitem{StacksProject}
The {Stacks Project Authors}, \emph{\textit{Stacks Project}}, \url{https://stacks.math.columbia.edu}, 2023.

\bibitem{Takahashi:2006_Glasgow}
Ryo Takahashi, \emph{Openness of {FID}-loci}, Glasg. Math. J. \textbf{48} (2006), no.~3, 431--435.

\bibitem{Verdier1977}
Jean-Louis Verdier, \emph{Catégories dérivées}, in SGA 4 1/2, Lecture Notes in Mathematics, vol. 569, Springer-Verlag, 1977, pp.~262--311. \MR{3727440}

\bibitem{Yekutieli}
Amnon Yekutieli, \emph{Duality and {Tilting} for {Commutative} {DG} {Rings}}, arXiv:1312.6411, \url{https://arxiv.org/abs/1312.6411}.

\bibitem{Yekutieli:book}
\bysame, \emph{Derived categories}, Camb. Stud. Adv. Math., vol. 183, Cambridge: Cambridge University Press, 2020.

\end{thebibliography}

\end{document}